%% file: 977.tex
\documentclass[a4paper]{amsart} 
\usepackage[ascii]{inputenc} 

\usepackage{amssymb}
\usepackage{mathrsfs}
\usepackage[hidelinks]{hyperref}

\input{mgrimes}

\begin{document}
\makeatletter\def\shfiuwefootnote{\gdef\@thefnmark{}\@footnotetext}\makeatother\shfiuwefootnote{Version 2026-06-09. See \url{https://shelah.logic.at/papers/977/} for possible updates.}

\title {Modules and Infinitary Logics \\
Sh:977}
\author {Saharon Shelah}

\dedicatory {Dedicated to R\"udiger G\"obel for this 70th birthday}

\address{Einstein Institute of Mathematics\\
Edmond J. Safra Campus, Givat Ram\\
The Hebrew University of Jerusalem\\
Jerusalem, 91904, Israel\\
 and \\
 Department of Mathematics\\
 Hill Center - Busch Campus \\ 
 Rutgers, The State University of New Jersey \\
 110 Frelinghuysen Road \\
 Piscataway, NJ 08854-8019 USA}
\email{shelah@math.huji.ac.il}
\urladdr{http://shelah.logic.at}
\thanks{
First version done April 2010. The author thanks Alice Leonhardt for the beautiful typing up to 2019, and Matt Grimes for the typing on the current version. The author would also like to thank the
Israel Science Foundation and German-Israeli Science Foundation 
for partial support
for the earlier version.  
For the later version 
We thank the ISF (Israel Science Foundation)
for partial support by grant 2320/23 (2023-2027).  
Paper Number 977.}





\date{June 8, 2026}

\begin{abstract}
We deal with Abelian groups and $R$-modules. We consider theories
in infinitary logic of
the form $\mathbb{L}_{\lambda,\theta}$ of such structures $M$ and prove they
have elimination of quantifiers up to positive existential formulas,
(so ones defining subgroups of some power of $M$). 
However, we demand that we expand by enough 
individual constants. Hence those theories are stable in the
appropriate sense and understood to some extent.

In 2026, John Baldwin pointed out a mistake
in the end of the proof of the main claim of Section 4, 
which is used in the theorem in Section 2, 
and made further requests for clarifications. 
Here this is corrected, in addition to other improvements. 
The error is corrected 
in three ways --- in Section 2 
we can use a weaker version of Section 4,
and in section 4 we also get the original result with more assumptions on the cardinal. 
Lastly, we provide a shorter and self-contained proof of the main theorem in \S2.
We can use the older version of Section 4, 
but then we use somewhat larger cardinals.
\end{abstract}

\maketitle
\numberwithin{equation}{section}
\setcounter{section}{-1}
\newpage

\section{Introduction}

Much is known on classes of $R$-modules and first order logic.
Szmielew \cite{Szm48} proved the decidability of the theory of
Abelian groups. In \cite{Szm55}, she proved an elimination of
quantifiers in the theory of Abelian groups up to Boolean combinations
of p.e.\ (\emph{positive existential}) formulas.

Eklof \cite{Ek71} proved the existence of universal homogeneous
$R$-models in $\lambda$ if $\lambda = \lambda^{< \gamma}$, where
$\gamma$ depends only on $R$. Fisher
improved this to saturated models of elementary classes (see his
review of \cite{Ek71}); this implies stability by a general theorem from \cite[\S0]{Sh:11} (or \cite[Ch.\,III]{Sh:c}).

Baur \cite{Bau76} proved that for the class of $R$-modules, any first-order formula 
is equivalent to a Boolean combination of positive existential formulas, and also proved the stability of $\Th(M)$ for $M$ an $R$-module.

We like to know for a given ring $R$ how complicated the 
class of $R$-modules which are
models of a sentence $\psi$ in an infinitary logic.

\bn
\begin{question}\label{y2}
Given a ring $R$, for the class $\Mod_R$ of left $R$-modules:

\mn
1) Does it have (for the logic $\bbL_{\lambda,\mu}$) a kind of 
elimination of quantifiers (say, up to some depth)?

\mn
2) Is it stable? (Say, no formula $\varphi(\bar x,\bar y) \in \bbL_{\infty,\infty}(\tau_R)$
linearly ordering an arbitrarily long sequence of tuples in some models of $\psi$?)

\mn
3) Can we define something like non-forking?
\end{question}

\bn
\begin{question}\label{y5}
Do we have a parallel of the main gap --- i.e.\ proving that
either every $M \in \Mod_\psi$ can be characterized by some
suitable cardinal invariants \underline{or} that there are many complicated
$M \in \Mod_\psi$?

Here we first show that for any $R$-module, in
$\bbL_{\lambda,\theta}(\tau_R)$ (or better,
$\bbL_{\infty,\theta,\gamma}(\tau_R)$ --- see \ref{z3}(3)), we have a version of
eliminating quantifiers up to positive existential formulas. 
\underline{However}, we add parameters. Second, by this we can
prove some versions and consequences of stability. More specifically:
\begin{enumerate}
    \item[$\bullet$] After expanding by enough individual constants, every formula in\\ $\bbL_{\infty,\theta,\gamma}(\tau_R)$ (see \ref{z3}(3)) is 
    equivalent to a Boolean combination of such positive existential formulas.
\sn
    \item[$\bullet$] The number of added individual constants is reasonable: $\le \beth_\gamma\big( |\tau|^{< \theta} \big)$.
\sn
    \item[$\bullet$] We have stability: i.e.\ no long sequences of linearly ordered $\lepref{\theta}$-tuples.
\sn
    \item[$\bullet$] $(\Lambda^\pe_{\eps,\alpha},2)$-indiscernible implies $\Lambda^\pe_{\eps,\alpha}$-indiscernible.
\sn
    \item[$\bullet$] Convergence follows (see Definition \ref{b8}).
\end{enumerate}
\end{question}

\bigskip
In 2025, this work was continued in a paper with Asgharzadeh and Golshani \cite{Sh:1246}.

We may use models with several \emph{sorts} --- that is, multiple distinct structures defined on them. E.g.,\ when constructing an $R$-module we need a set of objects which are the elements of the module, and a set of elements of the ring $R$, each with their own operations of addition and (scalar) multiplication. 
Hence when we need to disambiguate them, we will write something like $x +_\bfs y$, $x -_\bfs v$ for each sort separately. It makes no difference (see \ref{z15}).

\newpage

\section {Preliminaries}

\mn
\begin{notation}\label{z2}
Let $\theta^-$ be $\sigma$ if $\theta = \sigma^+$ and $\theta$ if $\theta$ is a limit cardinal.
\end{notation}

\bn
\begin{definition}\label{z3}
1) A vocabulary $\tau$ consists of function symbols (e.g.\ individual constants) 
and predicates (relation symbols). In addition the vocabulary generally assign
to each of them its arity (number of places) $\arity_\tau(-)$; here it can be an 
infinite ordinal. An individual constant is a 0-place function.

\mn
2) For a vocabulary $\tau$, we say $M$ is a $\tau$-structure \underline{when} it consists of:
\sn
\begin{enumerate}[(a)]
    \item $|M|$, the \emph{universe} of $M$; this is a non-empty set of the so-called elements of $M$.    
    However, we may write $a \in M$, $\bar a \in {}^\eps\!M$, $A \subseteq M$, instead of
    $a \in |M|$, $\bar a \in {}^\eps(|M|)$, etc.
\sn
    \item $F^M$, a function from ${}^\eps\!M$ to $M$ (possibly partial), for each function symbol $F$ from $\tau$. Here $\eps$ is the ordinal $\arity_\tau(F)$.
\sn
    \item $P^M \subseteq {}^\eps\!M$ (where $\eps$ is the ordinal $\arity_\tau(P)$), for $P$ a predicate from $\tau$.
\end{enumerate}
We may write $\tau_M = \tau(M) \defeq \tau$.

\mn
3) $\bbL_{\mu,\theta,\alpha}(\tau)$ is the set of formulas $\varphi(\bar x) \in \bbL_{\mu,\theta}$ (so $\lh(\bar x) < \theta$) of quantifier depth $<1+\alpha$.
\end{definition}

\bn
\begin{definition}
\label{z5}
1) We say $\tau$ is a $\theta$-\emph{additive} (or a $\theta$-Abelian) vocabulary 
\underline{when} $\tau$ has the two-place function symbols 
$x+y,x-y$, the individual constant 0, and the other predicates and 
function symbols have arity $< \theta$.

\mn
2) $M$ is a $\theta$-\emph{additive structure} (or model) \underline{when}:
\begin{enumerate}[(a)]
    \item $\tau_M$, (the vocabulary of $M$) is a $\theta$-additive vocabulary.
\sn
    \item $G_M \defeq (|M|,+^M,-^M,0^M)$ is an Abelian group.
\sn
    \item If $P \in \tau_M$ is an $\eps$-place predicate \underline{then} $P^M$ is a subgroup of $(G_M)^\eps$.
\sn
    \item If $F \in \tau_M \setminus \{+,-,0\}$ is an $\eps$-place function symbol \underline{then} $F^M$ is a partial $\eps$-place function from $M$ to $M$ and 
    $$
    \mathrm{graph}\big(F^M\big) \defeq \big\{\bar a \caret \big\LL F^M(\bar a)\big\RR : \bar a \in \dom(F^M) \big\}
    $$ 
    is a subgroup of $(G_M)^{\eps +1}$.
\end{enumerate}

\mn
3) If $\varphi(\bar x,\bar y)$ is a formula in the vocabulary $\tau$, $M$ is a $\tau$-model, and $\bar b \in {}^{\lh(\bar y)}\!M$, then we write 
$\varphi(M,\bar b) = \varphi({}^{\lh(\bar x)}\!M,\bar b)$ to mean the set
$$
\big\{ \bar a \in {}^{\lh(\bar x)}\!M : M \models \varphi[\bar a,\bar b] \big\}.
$$

\mn
4) We say $M$ is a $\theta$-\emph{affine structure} (or model) \underline{when}:
\begin{enumerate}
    \item [(a),(b)] As in part (2).
\sn
    \item [(c)] If $P \in \tau_M$ is an $\eps$-place predicate \underline{then} $P^M$ is an affine subset. (That is, $\bar a,\bar b,\bar c \in P^M \Rightarrow \bar a - \bar b + \bar c \in P^M$.)
\sn
    \item [(d)] As in part (2), but $\mathrm{graph}\big(F^M\big)$ is an affine subset.
\end{enumerate}
\end{definition}

\bn
\begin{remark}\label{z5.1}
1) Definition \ref{z5}(4) is not used in our theorem.

\mn
2) In \ref{z5}(4), we may replace $\{+,-,0\}$ by a three-place function symbol $F_\af$, and change clauses (a)-(d) as follows.
\begin{enumerate}[(a)]
    \item  $F_\af \in \tau_M$
\sn
    \item For any $b \in M$, the function  $(x,y) \mapsto F_\af^M(x,b,y)$ defines an additive group $G_{M,b}$ in which $b$ is the zero.
\sn
    \item As in \ref{z5}(4), but here `affine subset' means 
    $$
    \bar a,\bar b,\bar c \in P^M \Rightarrow G^M(\bar a, \bar b, \bar c) \defeq \LL F_\af^M(a_\zeta,b_\zeta,c_\zeta) : \zeta < \eps\RR \in P^M.
    $$

    \item If $H \in \tau$ is an $\eps$-place function symbol then $\mathrm{graph}\big(H^M\big)$ is an affine subset.
\end{enumerate}

\mn
3) We may also fix a sequence of predicates $\LL P_i,F_i : i < i_*\RR$ in $\tau$, where
\begin{enumerate}
    \item [(b)$'$] 
    \begin{itemize}
        \item $\LL P_i^M : i < i_*\RR$ is a partition {of $M$} (so $P_i$ is a unary predicate).
\sn
        \item $F_i^M$ is an affine operation on $P_i^M$ (so $F_i$ is a three-place function symbol).
    \end{itemize}
\end{enumerate}

\mn
4) In the affine framework above, \ref{z5}(4) can be treated in a perhaps more transparent way: just add more individual constants to $\tau_M$. That is:
\begin{enumerate}
    \item [$\boxplus$] For $\clI \subseteq M$, consider replacing $\tau_M$ by $\tau_{M,\clI} \defeq \tau_M \cup \{c_a : a \in \clI\}$, where the $c_a$-s are individual constants. $M[\clI]$ is $M$ expanded to a $\tau_{M,\clI}$-model by {$c_a^{M[\clI]} = a$}.

    Now the atomic formula `$x = c_a$' is not good for the ``additive" version (but is for the affine one).
\end{enumerate}
In this framework, $\Lambda_{\alpha,\eps}^\pe$ is as in \ref{a6}, but in Claim \ref{a8}:
\begin{itemize}
    \item If $\varphi(\bar x) \in \Lambda_{\alpha,\eps}^\pe$ for $\bbL(\tau_{M,\clI})$, then $\varphi(M)$ is an affine subset of ${}^\eps\!M$ --- i.e.\ closed under $\bar a - \bar b + \bar c$. (Or closed under $F_\af^M$ as in \ref{z5.1}(2).) 
\end{itemize}
\end{remark}

\bn
\begin{remark}\label{z6}
1) Fisher \cite{Fis77} defines and deals with ``Abelian structures" in other directions.

\mn
2) We use parentheses for formulas $\varphi(\bar x)$, but write (e.g.) $M \models \varphi[\bar a]$ for $\bar a \in {}^{\lh(\bar x)}\!M$.
\end{remark}

\mn
\begin{definition}\label{z8}
1) We consider an $R$-module $M$ as a $\tau(R)$-structure, where $\tau_R = \tau(R)$ 
is the vocabulary of $R$-modules. I.e.\ there are binary functions $x+y,x - y$, an
individual constant $0$, and unary function symbols $F_a$ 
(interpreted as multiplication by $a$ from the left) for every $a \in R$.

\mn
2) If $\bar x$ and $\bar y$ have length $\eps$, \underline{then} we let
$\bar x + \bar y = \LL x_\zeta + y_\zeta : \zeta < \eps \RR$ and 
$\bar x - \bar y = \LL x_\zeta - y_\zeta : \zeta < \eps \RR$;
similarly for $a \bar x$ with $a \in R$, and when we replace 
$\bar x$ and/or $\bar y$ by a member of ${}^\eps\!M$.
\end{definition}

\bn
\begin{observation}\label{z12}
$1)$ For any ring $R$, an $R$-module is an $\aleph_0$-additive structure in the vocabulary $\tau_R$.

\mn
$2)$ For a $\tau$-additive model $M$, for every function symbol or $\tau$-term $F(\bar x)$, we have
\begin{enumerate}
    \item[$(a)$] $M \models ``F(\bar a \pm \bar b) = F(\bar a) \pm F(\bar b)"$  
    
    (When $F^M$ is partial, this means that
    if two of the terms are well-defined then so is the third, and the equality holds.)
\sn
    \item[$(b)$] For $P$ a predicate from $\tau_M$, $M \models P(\bar a \pm \bar b)$ whenever 
    $M \models P(\bar a) \wedge P(\bar b)$.
\end{enumerate}
\end{observation}

\newpage

\section {Eliminating quantifiers}

\begin{context}
\label{a2}
1) $R$ is a fixed ring and $\tau = \tau_R$ (see \ref{z8}(1)), \underline{or}
$\tau$ is just a $\theta$-additive vocabulary (see \ref{z5}(1), \ref{z12}(1)). 

\mn
2) $\bfK$ is the class of $R$-modules \underline{or} of $\tau$-additive models.

\mn
3) $M,N$ will denote $R$-modules \underline{or} just $\tau$-additive models.

\mn
4) $\theta = \cf(\theta)$.
\end{context}

\bn
\begin{definition}\label{a6}
For $\eps < \theta$ and ordinal $\alpha$ (and $\tau$ as in \ref{a2}(1)), we shall define 
$$
\Lambda^\pe_{\alpha,\eps} = \Lambda^{\pe,\theta}_{\alpha,\eps} = \Lambda^{\pe,\theta}_{\alpha,\eps}(\tau)
$$ 
as a set of formulas $\varphi(\bar x)$ in $\bbL_{\infty,\theta}(\tau)$ 
--- in fact, in $\bbL_{\infty,\theta,\alpha}(\tau)$ --- with $\lh(\bar x) = \eps < \theta$. 
The construction will proceed by induction on the ordinal $\alpha$.

For $\zeta < \theta$, we write $\Lambda^\pe_{\alpha,\eps,\zeta}$ for the set of 
$\varphi = \varphi(\bar x,\bar y)$ with $\lh(\bar x) = \eps$ and $\lh(\bar y) = \zeta$ 
with $\varphi \in \Lambda^\pe_{\alpha,\eps + \zeta}$. We define
$\Lambda^\pe_\alpha \defeq \bigcup\limits_{\eps < \theta} \Lambda^\pe_{\alpha,\eps}$ and 
$\Lambda^\pe_{\alpha,\eps,< \theta} \defeq \bigcup\limits_{\zeta < \theta} \Lambda^\pe_{\alpha,\eps,\zeta}$. 
If $\tau = \tau_R$ we may write $\Lambda^\pe_{\alpha,\eps}(R)$.

The definition is as follows:

\bn
\textbf{Case 1:} $\alpha = 0$.

\mn
\underline{For $R$-modules:}

$\Lambda^\pe_{0,\eps}$ is the set of formulas
$\varphi = \varphi(\bar x)$ of the form
$\sum\limits_{\ell < n} a_\ell x_{\zeta_\ell} = 0$ with $\zeta_\ell < \lh(\bar x)$. 
Equivalently (but perhaps better phrased), they are of the form 
$\sum\limits_{\zeta < \eps} a_\zeta x_\zeta = 0$, where $a_\zeta \in R$ is $0_R$ 
for all but finitely many $\zeta$-s.

\mn
\underline{For }g\underline{eneral $\tau$:} (so here, this is the $\tau$-additive case).

$\Lambda^\pe_{0,\eps}$ is the set of formulas $\varphi(\bar x)$ of the form $P(\olsi\sigma(\bar x))$, 
where $\olsi\sigma$ is a sequence of length $\arity_\tau(P)$ of 
terms 
(in the variables $\bar x$). 
$P$ may be equality, or any predicate from $\tau$ of arity equal to $\lh(\olsi\sigma)$.

\bn 
\textbf{Case 2:} $\alpha$ a limit ordinal.

$\Lambda^\pe_{\alpha,\eps} \defeq \bigcup\limits_{\beta < \alpha} \Lambda^\pe_{\beta,\eps}$.

\bn
\textbf{Case 3:} $\alpha = \beta + 1$.

We define $\Lambda_{\alpha,\eps}^\pe$ as the union of $\Lambda_{\beta,\eps}^\pe$ 
together with the set of all formulas 
$\psi(\bar x)$ of the form 
$$
(\exists \bar y_\zeta)\bigwedge \big\{\varphi(\bar x \caret \bar y_\zeta) : \varphi(\bar x,\bar y_\zeta) \in \Phi_\zeta \big\},
$$ 
for some $\zeta < \theta$ and $\Phi_\zeta \subseteq \Lambda^\pe_{\beta,\eps, \zeta}$.
(Again, keep in mind $\Lambda^\pe_{\beta,\eps, \zeta}$ is a set of formulas of the form $\varphi(\bar x_{[\eps]},\bar y_{[\zeta]})$.)
\end{definition}

\bn
\begin{claim}
\label{a8}
$1)$ In \emph{\ref{a6}}, $\Lambda^\pe_{\alpha,\eps}$ is $\subseteq$-increasing with $\alpha$, and is of cardinality\\ $\le \beth_\alpha(|\tau| + \aleph_0)$ if 
$\theta = \aleph_0$, and $\beth_\alpha(|\tau|^{< \theta})$ in general.

\mn
$2)$ For $M \in \bfK$ and $\varphi(\bar x) \in \Lambda^\pe_{\alpha,\eps}(\tau)$, the set
$$
\varphi(\olsi M) \defeq \big\{ \bar b \in {}^\eps\!M : M \models \varphi[\bar b] \big\}
$$
is an Abelian subgroup of ${}^\eps\!M$, and the set $\big\{\bar b \in {}^\eps\!M : M \models \varphi[\bar b-\bar a]\big\}$ is affine (that is, closed under $\bar x-\bar y + \bar z$) for any $\bar a \in {}^\eps\!M$.
\end{claim}

\begin{PROOF}{\ref{a8}}
Easy.
\end{PROOF}

\bn
\begin{theorem}\label{a10}
For every $\alpha$ and every $M \in \bfK$, there is a subset $\bfI = \bfI_\alpha$ 
of ${}^{\theta >}\!M$ of cardinality 
$\le \kappa_\alpha \defeq \beth_\alpha(|\tau|^{< \theta})$ such that 
in $M$, we have
\begin{enumerate}
    \item [$\boxdot_\alpha$] Every formula $\psi(\bar x)$ from  $\bbL_{\infty,\theta,\alpha}(\tau)$ (so $\lh(\bar x) < \theta$) is equivalent in $M$ to a Boolean combination of formulas (possibly infinitely many) of the form $\varphi(\bar x - \bar a)$ with  $\varphi(\bar x) \in \Lambda^\pe_{\alpha,\lh(\bar x)}(\tau)$ and $\bar a \in \bfI \cap {}^{\lh(\bar x)}\! M$. 
\end{enumerate}
\end{theorem}

\bn
Before proving \ref{a10}, we shall note that it implies
\begin{conclusion}\label{a12}
For every $M \in \bfK$, $\eps < \theta$, $\bfI_\alpha$ as in Theorem \emph{\ref{a10}} for $\alpha$ a limit ordinal, and 
$\bar a \in {}^\eps\!M$, for some
$i_*,j_* \le \kappa_\alpha$ and $\varphi_i(\bar x),\psi_j(\bar x) \in \Lambda^\pe_{\alpha,\eps}$ for $i < i_*$, $j < j_*$ we have that
$$
\big\{ \bar a' \in {}^\eps\!M : \tp_{\bbL^\pe_{\infty,\theta,\alpha}}\!(\bar a',\varnothing,M) = \tp_{\bbL^\pe_{\infty,\theta,\alpha}}\!(\bar a, \varnothing,M) \big\}
$$ 
is equal to 
$$
\big\{\bar a' \in {}^\eps\!M : M \models \bigwedge\limits_{i < i_*} \varphi_i(\bar a' - \bar a) \wedge \bigwedge \{\neg \psi_j(\bar a' - \bar a'') : j < j_* \text{ and } \bar a'' \in \bfI_\alpha \cap {}^\eps\!M\} \big\}.
$$
\end{conclusion}

\bn
\begin{definition}\label{a17}
1) We say $\bar b_1,\bar b_2 \in {}^\eps\!M$ are
$\alpha$-\emph{equivalent over} $\bfI \subseteq {}^{\theta >}\!M$ \underline{when}
$$
\varphi(\bar x_{[\eps]}) \in \Lambda^\pe_{\alpha,\eps}(R) \wedge \bar a \in \bfI \cap {}^\eps\!M\  \Rightarrow\ M \models ``\varphi[\bar b_1 - \bar a] \Leftrightarrow \varphi[\bar b_2 - \bar a]".
$$

\mn
2) If we write $A \subseteq M$ instead of $\bfI$, we  mean $\bfI = {}^{\theta>}\!A$.
\end{definition}

\bn
We shall use freely
\begin{observation}\label{a19}
The sequences $\bar b_1,\bar b_2 \in {}^\eps\!M$ are
$\alpha$-equivalent over $\bfI \subseteq {}^\eps\!M$ \underline{iff}
for any $\varphi(\bar x)\in \Lambda^\pe_{\alpha,\eps}$ we have
$(a) \vee (b)$, where:
\begin{enumerate}
    \item[$(a)$] For some $\bar a \in \bfI \cap {}^\eps\!M$ we have $M \models \varphi[\bar b_1 - \bar a] \wedge \varphi[\bar b_2 - \bar a]$.
\sn
    \item[$(b)$] For every $\bar a \in \bfI \cap {}^\eps\!M$ we have $M \models \neg\varphi[\bar b_1 - \bar a] \wedge \neg\varphi[\bar b_2 - \bar a]$.
\end{enumerate}
\end{observation}

\begin{PROOF}{\ref{a19}}
Straightforward, recalling basic properties of cosets of Abelian groups.
\end{PROOF}

\sn
\begin{remark}\label{a22}
Note that the old proof of Theorem \ref{a10} relied on results in \S4, but more is proved there than is necessary here. Specifically, we just need Definition \ref{p0}(1) and clauses (a)+(b) of Claim \ref{p2}.

However, the proof presented below does not rely on \S4 at all.
\end{remark}

\mn
\begin{PROOF}[Proof of Theorem \emph{\ref{a10}}]{\ref{a10}}
We choose $\bfI_\alpha$ by induction on $\alpha$ such that 
\begin{enumerate}
    \item [$\circledast_\alpha$] 
    \begin{enumerate}
        \item $\bfI_\alpha \subseteq {}^{\theta>}\!M$ is of cardinality $\leq \kappa_\alpha$.
\sn
        \item Clause $\boxdot_\alpha$ from the statement of the theorem holds.
\sn
        \item ${}^{\theta>}\{0\} \subseteq \bfI_\alpha$
\sn
        \item {$\LL \bfI_\beta : \beta \leq \alpha\RR$ is $\subseteq$-increasing continuous.}
    \end{enumerate}
\end{enumerate}

For $\alpha = 0$ choose $\bfI_0 \defeq {}^{\theta>}\{0^M\}$, and for 
$\alpha$ a limit ordinal we obviously want $\bfI_\alpha \defeq \bigcup\limits_{\beta < \alpha} \bfI_\beta$. 

So assume $\alpha = \beta +1$ and $\bfI_\beta$ is given, and we shall choose $\bfI_\alpha$ such that
\mn
\begin{enumerate}
    \item[$\boxplus_{\beta+1}$]
    \begin{enumerate}
        \item $\bfI_\alpha$ is a subset of ${}^{\theta >}\!M$.
\sn
        \item $|\bfI_\alpha| \le 2^{\kappa_\beta}$ (recalling $\kappa_\beta \defeq \beth_\beta(|\tau|^{< \theta})$).
\sn
        \item $\bfI_\beta \subseteq \bfI_\alpha$
\sn
        \item If $\zeta < \theta$ and $\olsi\varphi = \LL \varphi_i(\bar x) : i < i_* \leq \kappa_\beta\RR$ 
        and $\olsi\varphi^\bullet = \LL \varphi_i^\bullet(\bar x) : i < i_\bullet\RR$ are sequences in $\Lambda_{\beta,\zeta}^\pe$, \underline{then} for some $\gamma_* \leq \kappa_\beta^+$ there exists $\LL\bar d_\gamma : \gamma < \gamma_*\RR \subseteq \bfI_\alpha$ with $\lh(\bar d_\gamma) \defeq \zeta$
        such that
        \begin{enumerate}
            \item $\bar d_\gamma \in \bigcap\limits_{i<i_*} \varphi_i(M)$ for all $\gamma < \gamma_*$.
\sn
            \item $\bar d_{\gamma_2} - \bar d_{\gamma_1} \notin \bigcup\limits_{i<i_\bullet} \varphi_i^\bullet(M)$ for all\footnote{
                \textbf{Note:} $\bullet_2$ does not exactly say that $\varphi_i^\bullet(M) \cap \bigcap\limits_{j<i_*} \varphi_j(M)$ has many cosets inside $\bigcap\limits_{i<i_*} \varphi_i(M)$. If $i_\bullet \defeq 1$ and $\gamma_* \defeq \kappa^+$ then yes, but otherwise it is more complicated.} 
            $\gamma_1 < \gamma_2 < \gamma_*$. 
\sn
            \item If $\gamma_* < \kappa_\beta^+$ and $\bar d \in \bigcap\limits_{i<i_*} \varphi_i(M)$ \underline{then} there exist $\gamma< \gamma_*$ and 
            $i < i_\bullet$ such that $\bar d - \bar d_\gamma \in \varphi_i^\bullet(M)$.
        \end{enumerate}       
    \end{enumerate}
\end{enumerate}

\smallskip
Why can we choose $\bfI_\alpha$? Let $\Xi$ be the set of triples $\bff = (\zeta,\olsi\varphi,\olsi\varphi^\bullet)$ satisfying the assumptions of clause (d). Clearly $\Xi$ has cardinality $\leq \big| \clP(\Lambda_{\beta,\eps+\eps}^\pe) \times \bfI_\alpha\big| \leq 2^{\kappa_\beta} + |\bfI_\alpha| = 2^{\kappa_\beta}$.

Rephrasing clause $\boxplus_{\beta+1}$(d) in this notation, we mean:
\begin{quotation}
    For every $\bff \in \Xi$ there are $\gamma_\bff \leq \kappa_\beta^+$ and $\gd_\bff = \LL\bar d_{\bff,\gamma} : \gamma < \gamma_\bff \RR \subseteq \bfI_\alpha$ such that the conclusions of clause (d) hold.
\end{quotation}
Let us fix $\bff = (\zeta_\bff,\olsi\varphi_\bff,\olsi\varphi_\bff^\bullet)$ and \underline{try} to choose $\bar d_{\bff,\gamma}$ by induction on $\gamma < \kappa_\beta^+$ to satisfy demands $\bullet_1$ and $\bullet_2$ of $\boxplus_{\beta+1}$(d).

Let $\gamma_\bff \leq \kappa_\beta^+$ be the first ordinal where we fail (or just where we stop). That is,

\sn
\begin{enumerate}
    \item [$\boxplus_\bff^1$] $\bar d_{\bff,\gamma}$ is defined \underline{iff} $\gamma < \gamma_\bff$.
\end{enumerate}
(Note that with these bounds, we will never try to define $\gamma_{\bff,\kappa_\beta^+}$.)

Let $\gd_\bff$ denote the sequence $\LL\bar d_{\bff,\gamma} : \gamma < \gamma_\bff \RR$. Note that

\sn
\begin{enumerate}
    \item [$\boxplus_\bff^2$] $\gamma_\bff \geq 1$ \underline{iff} $\bigcap\limits_{i<i_*} \varphi_i(M) \neq \varnothing$.
\end{enumerate}
[Why? Because bullet (d)$\,\bullet_2$ is vacuous when choosing $\bar d_{\bff,0}$, and $\bullet_1$ would be impossible to satisfy were the intersection empty.]

\smallskip
So for proving that $\gd_\bff$ is as required in $\boxplus_{\beta+1}$(d), we have to verify $\bullet_3$.

\mn
\textbf{Case 1:} $\gamma_\bff < \kappa_\beta^+$.

Here we know $\gd_\bff$ must also satisfy clause $\boxplus_{\beta+1}$(d)$\bullet_3$, as we cannot choose $\bar d_{\bff,\gamma_\bff}$.

In more detail: given some $\bar d \in \bigcap\limits_{i<i_*} \varphi_i(M)$, we have to prove that there exist $\gamma < \gamma_\bff$ and $i < i_\bullet$ such that
$\bar d - \bar d_\gamma \in \varphi_i^\bullet(M)$. Assume for the sake of a contradiction that there are no such $\gamma$ or $i$: in other words, 
$$
\bar d - \bar d_{\gamma_1} \notin \bigcup\limits_{i<i_\bullet} \varphi_i^\bullet(M)
$$ 
for any $\gamma_1 < \gamma_\bff$, which is precisely (d)\,$\bullet_2$.
$\bar d$ also satisfies bullet (d)\,$\bullet_1$ by assumption, and so $\gd_\bff \caret \LL\bar d\RR$ satisfies clause (d), contradicting $\boxplus_\bff^1$ (the maximality of $\gd_\bff$).

\mn
\textbf{Case 2:} $\gamma_\bff = \kappa_\beta^+$.

In this case we are done, as demand $\bullet_3$ is vacuous.

\medskip
So we have succeed in:
\begin{enumerate}
    \item [$\boxtimes_3$] 
    \begin{itemize}
        \item Choosing $\gd_\bff = \LL \bar d_{\bff,\gamma} : \gamma < \gamma_\bff\RR$ for every $\bff \in \Xi$.
\sn
        \item Proving that (d)$\bullet_1$-$\bullet_3$ holds for every $\bff \in \Xi$.
    \end{itemize}
\end{enumerate}
{Now let $\bfI_\alpha \defeq \bfI_\beta \cup \bigcup\limits_{\bff \in \Xi} \{\bar d_{\bff,\gamma} : \gamma < \gamma_\bff \}$; clearly it is as required in $\boxplus_{\beta+1}$.}

\bn
\centerline{*\qquad*\qquad*}

\medskip
At this point all we have done is prove that 
$\bfI_\alpha$ (for $\alpha  = \beta + 1$) can be chosen as required in $\boxplus_{\beta+1}$. Our job now is to prove that this 
$\bfI_\alpha$ satisfies the demands of $\circledast_\alpha$. 

Clauses $\circledast_\alpha$(a) and (d) are obvious, and clause (c) is satisfied by our construction, so we are left with $\circledast_\alpha$(b) --- equivalently, \ref{a10}$\boxdot_\alpha$. (Proving that we actually fulfill what was written in the theorem.)  

To this end, clearly it suffices to prove the following.
\begin{enumerate}
    \item[$\boxtimes_4$] Assume $\eps,\xi < \theta$, $\bar b_1,\bar b_2 \in {}^\eps\!M$ are 
    $\alpha$-equivalent over $\bfI_\alpha$, and $\bar c_1 \in {}^\xi\!M$. 
    \underline{Then} for some $\bar c_2 \in {}^\xi\!M$ the sequences $\bar b_1 \caret \bar c_1$ and $\bar b_2 \caret \bar c_2 \in {}^{\eps + \xi}\!M$ are $\beta$-equivalent over $\bfI_\beta$.
\end{enumerate}

Let $\olsi\varphi = \LL \varphi_i(\bar x, \bar y) : i < i_*\RR$ list the formulas in the set 
\begin{align*}
    \Phi_1 \defeq \big\{ \varphi(\bar x,\bar y) \in \Lambda_{\beta,\eps,\xi}^\pe : &\text{ there exists $\bar b \caret \bar c \in \bfI_\alpha$ such that } \lh(\bar b) = \eps,\\
    &\ \lh(\bar c) = \xi,
    \text{ and } M \models \varphi(\bar b_1 - \bar b,\bar c_1 - \bar c) \big\}.
\end{align*}
(Note that if $M \models \varphi[\bar b_1,\bar c_1]$ then $\varphi(\bar x,\bar y) \in \Phi_1$ as we can choose $\bar b \defeq \bar 0_\eps$ and $\bar c \defeq \bar 0_\zeta$.)

\medskip
Similarly, let $\olsi\varphi^\bullet = \LL \varphi_i^\bullet(\bar x, \bar y) : i < i_\bullet\RR$ list the formulas in $\Phi_2 \defeq \Lambda_{\beta,\eps,\xi}^\pe \setminus \Phi_1$. 
Clearly, without loss of generality $i_*$ and $i_\bullet$ are ordinals $\leq \kappa_\beta$, as $\big| \Lambda_{\beta,\eps,\xi}^\pe \big| = \kappa_\beta$.

Let $\zeta \defeq \eps+\xi$, and let $\LL \bar d_\gamma : \gamma < \gamma_*\RR$ be as guaranteed to exist by clause $\boxplus_{\beta+1}$(d),
for $\bff = (\zeta,\olsi\varphi,\olsi\varphi^\bullet) \in \Xi$ as above.

Let $\bar b_\gamma^\bullet \defeq \bar d_\gamma \rest \eps$ and $\bar c_\gamma^\bullet \defeq \bar d_\gamma \rest [\eps,\eps+\xi)$.

\mn
\textbf{Case 1:} $\gamma_* < \kappa_\beta^+$.

Let $\bar d \defeq \bar b_1 \caret \bar c_1$, so clearly $\bar d \in \bigcap\limits_{i<i_*} \varphi_i(M)$. By clause $\boxplus_{\beta+1}$(d)$\bullet_3$ we have $\bar d - \bar d_\gamma \in \varphi_i^\bullet(M)$ for some $\gamma < \gamma_*$ 
and $i < i_\bullet$. As $\bar d_\gamma \in \bfI_\alpha$, this implies that $\varphi_i^\bullet \in \Phi_1$. But we know $\varphi_i^\bullet \in \Phi_2$ by our definition of $\olsi\varphi^\bullet$, a contradiction.

\mn
\textbf{Case 2:} $\gamma_* = \kappa_\beta^+$.

As $\bar b_1$ and $\bar b_2$ are $\alpha$-equivalent (see \ref{a17}, \ref{a19}) there exists $\bar c_2' \in {}^\xi\!M$ such that $M \models \bigwedge\limits_{i<i_*} \varphi_i(\bar b_2,\bar c_2')$. By our choice of the $\bar d_\gamma$-s, for every $\gamma < \gamma_*$ we have 
$M \models \bigwedge\limits_{i<i_*} \varphi_i(\bar b_\gamma^\bullet,\bar c_\gamma^\bullet)$, hence
$$
M \models \bigwedge\limits_{i<i_*} \varphi_i[\bar b_2 - \bar b_\gamma^\bullet + \bar b_1,\bar c_2' - \bar c_\gamma^\bullet + \bar c_1].
$$

For each $i < i_\bullet$ the set 
$$
\clS_i \defeq \big\{ \gamma < \gamma_* : M \models \varphi_i^\bullet[\bar b_2 - \bar b_\gamma^\bullet + \bar b_1,\bar c_2' - \bar c_\gamma^\bullet + \bar c_1] \big\}
$$ 
has at most one member (by $\boxplus_{\beta+1}$(d)$\bullet_2$ plus some algebra).

\sn
[In more detail: towards contradiction, assume $\gamma_1 \neq \gamma_2$ are in $\clS_i$, so
$$
M \models \varphi_i^\bullet[\bar b_2 - \bar b_{\gamma_\ell}^\bullet + \bar b_1,\bar c_2' - \bar c_{\gamma_\ell}^\bullet + \bar c_1]
$$
for $\ell = 1,2$. 

But $\varphi_i^\bullet(M)$ is closed under subtraction, and 
$$
(\bar c_2' - \bar c_{\gamma_1}^\bullet + \bar c_1) - (\bar c_2' - \bar c_{\gamma_2}^\bullet + \bar c_1) = \bar c_{\gamma_2}^\bullet - \bar c_{\gamma_1}^\bullet
$$
and
$$
(\bar b_2 - \bar b_{\gamma_1}^\bullet + \bar b_1) - (\bar b_2 - \bar b_{\gamma_2}^\bullet + \bar b_1) = \bar b_{\gamma_2}^\bullet - \bar b_{\gamma_1}^\bullet.
$$
Hence $M \models \varphi_i^\bullet[\bar b_{\gamma_2}^\bullet - \bar b_{\gamma_1}^\bullet,\bar c_{\gamma_2}^\bullet - \bar c_{\gamma_1}^\bullet]$, contradicting $\boxplus_{\beta+1}$(d)$\bullet_2$.]

\mn
Hence 
$$
\big|\textstyle\bigcup\limits_{i<i_\bullet} \clS_i\big| \leq |i_\bullet| \leq i_\bullet \leq \kappa_\beta < \gamma_* \defeq \kappa_\beta^+,
$$
so there exists $\gamma \in \gamma_* \setminus \bigcup\limits_{i<i_\bullet} \clS_i$. Now $\bar c_2' - \bar c_\gamma^\bullet + \bar c_1$ is as required in clause $\boxtimes_4$.
\end{PROOF}

\bn
\centerline{*\qquad*\qquad*}

\mn
\begin{PROOF}[\textsc{Earlier Proof of \ref{a10}:}]{\ref{a10}OLD}\label{a25}
We choose $\bfI_\alpha$ by induction on $\alpha$ such that 
\begin{enumerate}
    \item [$\circledast_\alpha$] 
    \begin{enumerate}
        \item $\bfI_\alpha \subseteq {}^{\theta>}\!M$ is of cardinality $\leq \kappa_\alpha$.
\sn
        \item $\big(\forall \eps < \theta\big) \big(\forall \varphi(\bar x) \in \Lambda_{\alpha,\eps}^\pe(\tau) \big) \big(\forall \bar a \in {}^\eps\!M\big)  \big(\exists \bar b \in \bfI_\alpha \cap  {}^\eps\!M\big) \big[ M \models \varphi[\bar b - \bar a]\big]$
\sn
        \item $\LL \bfI_\beta : \beta \leq \alpha\RR$ is $\subseteq$-increasing continuous.
\sn
        \item $\bfI_\alpha$ satisfies the demands of the theorem (i.e.\ $\boxdot_\alpha$ holds).
    \end{enumerate}
\end{enumerate}

For $\alpha = 0$ choose $\bfI_0 \defeq {}^{\theta>}\{0^M\}$, and for 
$\alpha$ a limit ordinal we obviously want $\bfI_\alpha \defeq \bigcup\limits_{\beta < \alpha} \bfI_\beta$. 

So assume $\alpha = \beta +1$ and $\bfI_\beta$ is given, and we shall choose $\bfI_\alpha$ such that
\mn
\begin{enumerate}
    \item[$\boxplus_{\beta+1}$]
    \begin{enumerate}
        \item $\bfI_\alpha$ is a subset of ${}^{\theta >}\!M$.
\sn
        \item $|\bfI_\alpha| \le 2^{\kappa_\beta}$, where $\kappa_\beta \defeq \beth_\beta(|\tau|^{< \theta})$.
\sn
        \item $\bfI_\beta \subseteq \bfI_\alpha$
\sn
        \item If $\eps < \theta$, $\varphi_i(\bar x) \in \Lambda^\pe_{\beta,\eps}$, 
        $\bar a_i \in \bfI_\beta \cap {}^\eps\!M$ for $i <i_* \le \kappa_\beta$, and there is $\bar d \in {}^\eps\!M$ such that $M \models \bigwedge\limits_{i<i_*}\varphi_i[\bar d - \bar a_i]$,
        \underline{then} there is such $\bar d \in \bfI_\alpha$.
\sn
        \item Assume $\eps < \theta$, $\lh(\bar x) = \eps$, $\psi(\bar x)$ is a conjunction of formulas from $\Lambda^\pe_{\beta,\eps}$ and $\varphi_i(\bar x) \in \Lambda^\pe_{\beta,\eps}$ for $i < \kappa_\beta$. 

        Let $G$ be the Abelian group with set of elements 
        $$
        \psi({}^\eps\!M) = \big\{\bar a \in {}^\eps\!M : M \models \psi[\bar a] \big\}.
        $$
        and addition of two such sequences defined coordinatewise. Let $\olsi\varphi = \LL \varphi_i : i < \kappa_\beta\RR$ and $G_i \defeq \varphi_i(M) \cap \psi(M)$ (so it is a subgroup of $G$).
        Letting $\mu \defeq 2^{\kappa_\beta}$, we apply \ref{p2} with $\lambda \defeq \mu^+$, $S \defeq \kappa_\beta$, and 
        $\olsi G \defeq \LL G_i : i \in S\RR$ (although for this proof we will only need clauses $(a)$ and $(b)$ of the conclusion).
        
        This gives us a certain family of sets $I \subseteq \clP(S)$ -- not necessarily an ideal. Further assume that $S \defeq \kappa_\beta \notin I$. \underline{Then}
        \begin{enumerate}
            \item There are $\bar d_\iota \in \bfI_\alpha \cap \varphi_i({}^\eps\!M)$ for  
            $\iota < \iota_* \le \mu$ such that for every 
            $\bar a \in G \defeq \psi({}^\eps\!M)$ there exists $\iota < \iota_*$ such that 
            $$
            \big\{j \in S : \bar a - \bar d_\iota \notin \varphi_j({}^\eps\!M) \big\} \in I.
            $$
            \item For any $u \in I$ there is a set $u_*$ with $u \subseteq u_* \in I$ and a sequence
            $$
            \LL\bar d_\iota : \iota < \mu\RR \subseteq \bigcap\{\varphi_i({}^\eps\!M) : i \in S \setminus u_*\} \cap \psi({}^\eps\!M) \cap \bfI_\alpha
            $$ 
            such that 
$$
            \big( \forall i < \kappa_\beta \big) \big( \forall \iota_1 \le \iota_2 < \mu \big) \big[ \bar d_{\iota_1} - \bar d_{\iota_2} \notin \varphi_i({}^\eps\!M) \Leftrightarrow i \in u_* \big].
$$
        \end{enumerate}
\sn
        \item If $\eps < \theta$ and $\bar d_1,\bar d_2 \in \bfI_\alpha \cap {}^\eps\!M$ \underline{then} $\bar d_1 + \bar d_2 \in \bfI_\alpha$, $\bar d_1 - \bar d_2 \in \bfI_\alpha$, and  
        $\xi < \theta \Rightarrow \bar 0_\xi \caret \bar d_1 \in \bfI_\alpha$.
    \end{enumerate}
\end{enumerate} 

\mn
Why can we choose such $\bfI_\alpha$?
It suffices to show that each of the demands above hold for a club's worth of sets in $[{}^\eps\!M]^\mu$. 

For reference, recall
\begin{itemize}
    \item For a set $X$ and cardinal $\Upsilon$, $E$ is a club of $[X]^{\leq\Upsilon}$ \underline{iff} there exist functions $F_{\xi,n} : X \to X$ for $\xi < \Upsilon$ and $n < \omega$ {such that}
$$
    E = \{ u \in [X]^{\leq\Upsilon} : u \text{ is closed under the $F_{\xi,n}$-s}\}.
$$
    \item If $E_i$ is a club of $[X]^{\leq\Upsilon}$ for $i < i_* \leq \Upsilon$ then so is $\bigcap\limits_{i<i_*} E_i$.
\end{itemize}

Now $\boxplus_{\beta+1}$(a),(b) are obvious.

This holds for clause (c) by the induction hypothesis, as we will choose $\bfI_\alpha$ by adding elements to $\bfI_\beta$.

This holds for clause (d) because  $\Lambda^\pe_{\beta, \eps }$ has cardinality $\kappa_\beta$ and $\mu = 2^{\kappa_\beta}$.

This holds for (e)$\bullet_1$ by clause $(a)$ of Claim \ref{p2}and for (e)$\bullet_2$ by \ref{p2}$(b)$.

\bigskip
To prove the induction statement for $\alpha$, clearly it suffices
to prove the following.
\begin{enumerate}
    \item[$\boxdot$] Assume $\eps,\xi < \theta$, $\bar b_1,\bar b_2 \in {}^\eps\!M$ are 
    $\alpha$-equivalent over $\bfI_\alpha$, and $\bar c_1 \in {}^\xi\!M$. 
    \underline{Then} for some $\bar c_2 \in {}^\xi\!M$ the sequences $\bar b_1 \caret \bar c_1$ and $\bar b_2 \caret \bar c_2 \in {}^{\eps + \xi}\!M$ are $\beta$-equivalent over $\bfI_\beta$. (Here we rely on \ref{a19}.)
\end{enumerate}

\mn
Why does $\boxdot$ hold? Let $\bar x$ be of length $\eps$ and $\bar y$ of length $\xi$.
Let 
$$
\Phi_1 \defeq \big\{ \varphi(\bar x,\bar y) \in \Lambda^\pe_{\beta,\eps + \xi} : \text{for some } \bar a \in \bfI_\beta \cap {}^{\eps + \xi}\!M \text{ we have } 
M \models \varphi[\bar b_1 \caret \bar c_1 - \bar a] \big\}.
$$ 
For $\varphi(\bar x,\bar y) \in \Phi_1$, we know (by the induction hypothesis) that $\circledast_\beta$ holds; hence we can choose 
$\bar a_{\varphi(\bar x,\bar y)} \in \bfI_\beta \cap {}^{\eps + \xi}\!M$ such that 
$M \models \varphi[\bar b_1 \caret \bar c_1 - \bar a_{\varphi(\bar x,\bar y)}]$. Let $\Phi_2 \defeq \Lambda^\pe_{\beta,\eps + \xi} \setminus \Phi_1$.

So by $\boxplus_{\beta+1}$(d) there is a sequence 
$\bar b^* \caret \bar c^* \in \bfI_\alpha$ such that $\lh(\bar b^*) = \lh(\bar b_1)$, $\lh(\bar c^*) = \lh(\bar c_1)$, and
$\varphi(\bar x,\bar y) \in \Phi_1 \Rightarrow M \models \varphi[\bar b^* \caret \bar c^* - \bar a_{\varphi(\bar x,\bar y)}]$. 
For transparency, note that if $\Phi_2 = \varnothing$ then 
(as the formula $(\exists \bar y)\bigwedge\limits_{\varphi \in \Phi_1} \varphi(\bar x,\bar y)$ is a member of $\Lambda^\pe_{\alpha,\eps + \xi}$) clearly by the assumption
of $\boxdot$ there is $\bar c_2 \in {}^\xi\!M$ such that 
$$
\varphi(\bar x,\bar y) \in \Phi_1 \Rightarrow M \models \varphi(\bar b_2 \caret \bar c_2 - \bar a_{\varphi(\bar x,\bar y)}).
$$
So $\bar c_2$ is as required, hence we are done. 

So without loss of generality $\Phi_2 \ne \varnothing$. 
Clearly $|\Phi_2| \le \kappa_\beta$, and let 
$$
\Phi'_\ell \defeq \{\varphi(\bar 0_\eps,\bar y) : \varphi(\bar x,\bar y) \in \Phi_\ell\}
$$ 
for $\ell=1,2$.

Let $\{\neg \varphi_i(\bar x \caret \bar y - \bar a_i) : i < \kappa_\beta\}$ list 
(possibly with repetitions) 
the set of formulas $\neg \varphi(\bar x \caret \bar y - \bar a)$ satisfied by $\bar c_1 \caret \bar b_1$ with $\bar a \in \bfI_\beta$ and
$\varphi(\bar x,\bar y) \in \Lambda^\pe_{\beta,\eps,\zeta}$ (equivalently, $\varphi(\bar x,\bar y) \in \Phi_2$).
Let $\varphi'_i(\bar y) \defeq \varphi_i(0_\eps,\bar y)$
for $ i < \kappa_\beta $, and 
$\psi'(\bar y) \defeq \bigwedge\limits_{\varphi \in \Phi'_1} \varphi(\bar y)$.

As in $\boxplus_{\beta+1}$, let $S \defeq \kappa_\beta$ and $I = I_\lambda \subseteq \clP(S)$ be defined as in Definition \ref{p0}, with 
$G \defeq \psi'({}^\xi\!M)$, 
$G_i \defeq G \cap \varphi'_i({}^\xi\!M)$ for $i \in S$, and $\lambda \defeq \mu^+ = (2^{\kappa_\beta})^+$.

\bn
\textbf{Case 1:} $S \defeq \kappa_\beta \in I$.

So clearly $M \models \varphi[\bar b_1 - \bar b^*,\bar c_1 - \bar c^*]$ 
for every $\varphi(\bar x,\bar y) \in \Phi_1$.

Let $\psi_*(\bar x,\bar y) = \bigwedge\{\varphi(\bar x,\bar
y) : \varphi(\bar x,\bar y) \in \Phi_1\}$; clearly it is a member of 
$\Lambda^\pe_{\alpha,\eps,\zeta}$ and 
$M \models \psi_*[\bar b_1,\bar c_1]$. Hence by the choice of $(\bar b^*,\bar c^*)$ we also have $M \models \psi_*[\bar b^*,\bar c^*]$. 
As $\psi_*$ is positive existential, clearly
$M \models \psi_*[\bar b_1 - \bar b^*,\bar c_1 - \bar c^*]$, hence 
$$
M \models (\exists \bar y)\psi_*[\bar b_1 - \bar b^*,\bar y].
$$
But $(\exists \bar y) \psi(\bar x,\bar y) \in \Lambda^\pe_{\alpha,\eps}$, so by the
assumption on $\bar b_1$ and $\bar b_2$ we have 
$$
M \models (\exists \bar y)\psi_*[\bar b_2 - \bar b^*,\bar y],
$$
hence for some $\bar c'_2$ we have $M \models \psi_*[\bar b_2 - \bar b^*,\bar c'_2]$. 
Let $\bar c''_2 \defeq \bar c'_2 + \bar c^*$, so 
$$
M \models \psi_*[\bar b_2 - \bar b^*,\bar c''_2 - \bar c^*].
$$
By $\boxplus_{\beta+1}$(e)$\bullet_2$ and our assumption that $\kappa_\beta \in I$,
there is a sequence $\LL \bar e_\iota : \iota < \mu \RR$ of members of $G$ 
(i.e.\ of $\big\{ \bar a \in {}^\xi\!M : M \models \psi_*(\bar 0_\eps,\bar a) \big\}$), 
recalling that $ \psi '( \bar{ y }= \psi (0_\eps , \bar{ y })  $, 
such that 
$$
i < \kappa_\beta \wedge (\iota_1 < \iota_2 < \mu )\ \Rightarrow\ \bar e_{\iota_2} - \bar e_{\iota_1} \notin G_i.
$$

So for every $\iota < \mu $, the sequence 
$(\bar b_2 - \bar b^*) \caret (\bar c''_2 - \bar c^* + \bar e_\iota)$ belongs to
$\psi_*({}^{\eps + \xi}\!M)$ and for each $i < \kappa_\beta$ the
set $\{\iota < \mu :(\bar b_2 - \bar b^*) \caret (\bar c''_2 - \bar c^* + \bar e_\iota)$
belongs to $(\bar a_i - \bar b^* \caret \bar c^*) + G_i\}$ has at
most one member. As $\kappa_\beta < \mu $, we have 
$$
(\bar b_2 - \bar b^*) \caret (\bar c''_2 - \bar c^* + \bar e_\iota) \notin \bigcup \big\{ (\bar a_i - \bar b^*) \caret (\bar c^* + G_i) : i < \kappa_\beta \big\}
$$ 
for\footnote{
    Recall that $\bar b + G_i$ just means $\{\bar b + \bar a : \bar a \in G_i\}$.
} 
some $\iota < \mu$.

So $\bar c_2 \defeq \bar c''_2 + \bar e_\iota$ is as required.

\bn
\textbf{Case 2:} $S \notin I$.

So there is a sequence $\LL \bar d_\iota : \iota < \iota_*\RR$
of members of $\bfI_\alpha$ as
in $\boxplus_{\beta+1}$(e)$\bullet_1$ for $\xi,G$, $\LL G_i : i < \kappa_\beta\RR$ as above
(i.e.\ with $\bar\psi'(\bar y)$, $\LL \varphi'_i(\bar y) : i < \kappa_\beta\RR$ here standing in for $\psi(\bar x),\LL \varphi_i(\bar x) : i < \kappa_\beta\RR$ there). So $\iota_* < (2^{\kappa_\beta})^+ = \mu^+$ 
and 
$\iota < \iota_* \Rightarrow \bar d_\iota \in \bfI_\alpha \cap {}^\xi\!M$. As clearly
$\bar c_1 - \bar c^* \in G$, necessarily for some $\iota < \iota_*$ the set 
$$
u \defeq \big\{ i < \kappa_\beta : (\bar c_1 - \bar c^* - \bar d_\iota) \notin G_i \big\}
$$ 
belongs to $I$ (and of course, $\bar b^* \caret (\bar c^* + \bar d_\iota) \in \bfI_\alpha \cap {}^{\eps + \xi}\!M$) and we have:
\sn
\begin{enumerate}
    \item[$(*)_1$] $M \models \varphi[\bar b_1 - \bar b^*,\bar c_1 - \bar c^* - d_\iota]$ for $\varphi \in \Phi_1$.
\sn
    \item[$(*)_2$] If $i \in \kappa_\beta \setminus u$ then $M \models \varphi_i[\bar b_1 - \bar b^*,\bar c_1 - \bar c^* - \bar d_\iota]$.
\end{enumerate}

\mn
As in Case 1, there is $\bar c''_2 \in {}^\xi\!M$ such that
\begin{enumerate}
    \item[$(*)_3$] $M \models \varphi[\bar b_2 - \bar b^*,\bar c''_2 - \bar c^* - \bar d_\iota]$ for $\varphi \in \Phi_1$.
\end{enumerate}

\sn
Hence
\begin{enumerate}
    \item[$(*)_4$] If $i \in \kappa_\beta \setminus u$ then 
    $M \models \varphi_i[\bar b_2 - \bar b^*,\bar c''_2 - \bar c^* - \bar d_\iota]$.
\end{enumerate}

As $u \in I$ by $\boxplus_{\beta+1}$(e)$\bullet_2$ (that is, by \ref{p2}) there
are $\bar{\bfe} = \LL \bar e_j : j < \mu\RR$ and $u_*$ with 
$u \subseteq u_* \in I$ such that:
\begin{enumerate}
    \item[$(*)_5$] $\{ \bar e_j : j < \mu\} \subseteq \bigcap\limits_{i \in \kappa_\beta \setminus u_*} G_i$.
\sn
    \item[$(*)_6$] $\bar e_{j_2} - \bar e_{j_1} \notin G_i$ for all $j_1 < j_2 < \mu$ and $i \in u_*$.
\end{enumerate}

\mn
So
\sn
\begin{enumerate}
    \item[$(*)_7$] If $j < \mu$ then $(\bar b_2 - \bar b^*) \caret (\bar c''_2 - \bar c^* - \bar d_\iota - \bar e_j)$ belongs to $\bigcap\limits_{\varphi \in \Phi_1} \varphi({}^{\eps + \xi}\!M)$.
\sn
    \item[$(*)_8$] If $i \in \kappa_\beta \setminus u_*$ then $i \in \kappa_\beta \setminus u$ as well, so by $(*)_4 $+$ (*)_5$ the sequence 
    $$
    (\bar b_2 - \bar b^*) \caret (\bar c''_2 - \bar c^* - \bar d_\iota - \bar e_j)
    $$ 
    satisfies $\varphi_i(\bar x \caret \bar y - \bar a_i)$ in $M$, hence $\bar b_2 \caret (\bar c''_2 - \bar e_j)$ satisfies the formula\\ $\neg \varphi_i(\bar x \caret \bar y - \bar a_i)$ in $M$.
\end{enumerate}
\mn
Lastly, by $(*)_6$,
\mn
\begin{enumerate}
    \item[$(*)_9$] For each $i \in u_*$ there is $j_i < \mu$ such that for every $j \in \mu \setminus \{j_i\}$, the sequence 
    $(\bar b_2 - \bar b^*) \caret (\bar c''_2 - \bar c^* - \bar d_\iota - \bar e_j)$ satisfies 
    $\neg \varphi_i(\bar x \caret \bar y - \bar a_i)$.
\sn
    \item[$(*)_{10}$] Moreover, 
    \begin{enumerate}
        \item The set $\mu \setminus \{j_i : i \in u_*\}$ is non-empty.
\sn
        \item For some {(equivalently, `for every')} $j$ in this set we have
$$
        (\forall i \in u_*) \neg\varphi_i\big[ (\bar b_2 - \bar b^*) \caret (\bar c''_2 - \bar c^* - \bar d_\iota - \bar e_j) - \bar a_i \big].
$$
    \end{enumerate}  
\end{enumerate}
[Why? Clause (a) is true simply because $\mu \defeq 2^{\kappa_\beta} > \kappa_\beta \geq |u_*|$, and clause (b) follows from $(*)_9$.]

\medskip
Putting together $(*)_7 $--$(*)_{10}$, clearly $(\bar c''_2 - \bar
c^* - \bar d_\iota - \bar e_j)$ is as required in $\boxdot$, so we are done.
\end{PROOF}




\bn
\begin{conclusion}\label{a24}
If $\lambda$ is a fixed point of the beth sequence (i.e.\ $\lambda = \beth_\lambda$) \underline{then} for $\tau$ a $\theta$-additive vocabulary for $M \in \bfK$, every formula $\varphi(\bar x) \in \bbL_{<\lambda,<\lambda}(\tau)$ (i.e.\ $\bigcup\limits_{\theta<\lambda} \bbL_{\theta,\theta}(\tau)$) is equivalent to a Boolean combination of positive existential formulas (from $\bbL_{\theta,\theta}$ for some $\theta = \theta_\varphi < \lambda$).
\end{conclusion}

\newpage

\section {Stability}

\begin{context}\label{b2}
1) $R$ is a fixed ring with $\tau = \tau_R$, \underline{or}
$\tau$ is a $\theta$-additive vocabulary; $\bfK$ is the
class of $\tau$-additive models.

\mn
2) $M$ will denote a member of $\bfK$. 

\mn
3) $\theta = \cf(\theta)$ and $\gamma^*$ is an ordinal --- limit, for simplicity.

\mn
4) $\bar \lambda = \LL \lambda_\alpha : \alpha \le \gamma^*\RR$, where 
$\lambda_\alpha > \kappa_\alpha \defeq \beth_\alpha\big(|R| + \theta^-\big)$.

\mn
5) $\bar{\bfI}^*$ is a $(\bar\lambda,\theta,\gamma^*)$-witness (that is, $\bfI_\alpha^*$ satisfies the conditions of \ref{a10} and $|\bfI_\alpha^*| < \lambda_\alpha$ for each $\alpha < \lh(\bar{\bfI}^*) \defeq \gamma_*$).

\mn
6) $A_* \defeq \bigcup\{\bar a : \bar a \in \bfI_{\gamma^*}\}$.

\mn
7) $\Lambda_\eps \defeq \Lambda^\pe_{\gamma^*,\eps}$ for $\eps < \theta$, and 
$\Lambda \defeq \bigcup\limits_{\eps < \theta} \Lambda_\eps$.

\mn
8) $M_* = M_{A_*} \defeq (M,a)_{a \in A_*}$.
\end{context}

\bn
\begin{definition}\label{b6}
Assume $\eps < \theta$, $\Lambda \subseteq \Lambda^\pe_{\theta,\gamma^*}$, 
$A_* \subseteq A \subseteq M \in \bfK$, and $\bar a \in {}^\eps\!M$.

\mn
1) For $\bar a \in {}^\eps\!M$, let 
\begin{align*}
    \tp_\Lambda\!(\bar a,A,M) \defeq \big\{ \varphi(\bar x \caret \bar b - \bar c) : &\ \bar b \in {}^\xi\!A,\ \bar c \in {}^{\eps + \xi}\!M,\ M \models \varphi[\bar a_1 \caret \bar b - \bar c],\\
    &\text{ and } \varphi(\bar x,\bar y) \in \Lambda^\pe_{\gamma,\eps + \xi} \cap \Lambda \big\}.
\end{align*}

\mn
2) $\bfS^\eps_\Lambda(A,M) \defeq \{\tp_\Lambda(\bar a,A,M) : \bar a \in {}^\eps\!M\}$.
\end{definition}

\bn
\begin{theorem}[\textbf{The Stability Theorem}]\label{b7}
Assume $\Lambda \subseteq \Lambda^\pe_{\gamma^*}$ and $A \subseteq M \in \bfK$.

\mn
$1)$ The set $\bfS^\eps_\Lambda(A,M)$ has cardinality 
$\le \big(|A|^{< \theta} \big)^{|\Lambda|}$.

\mn
$2)$ For any $\kappa \ge 4$ (yes, \emph{four}!) there are no 
$\LL\bar a_\alpha : \alpha < \kappa\RR \subseteq {}^\eps\!M$, $\LL\bar b_\alpha : \alpha < \kappa\RR \subseteq {}^\zeta\! M$ and\footnote{
    This also holds for $\neg \varphi(\bar x,\bar y)$, but for $\kappa$ finite we can invert the order.
}
$\varphi(\bar x,\bar y) \in \Lambda^\pe_{\gamma^*,\eps,\xi}$ such that for 
$\alpha < \beta < \kappa$ we have 
$$
M \models ``\varphi[\bar a_\alpha,\bar b_\beta] 
\wedge \neg \varphi[\bar a_\beta,\bar b_\alpha]".
$$

\mn
$3)$ If the formula $\varphi(\bar x,\bar y)$ is from $\bbL_{\infty,\theta,\gamma^*}$ 
(or is just a Boolean combination of such formulas) and 
$\kappa \ge \beth_{\gamma^*+2}(|\tau|^{< \theta})^+$, 
\underline{then} there are no $M \in \bfK$, $\LL\bar a_\alpha : \alpha < \kappa\RR \subseteq {}^\eps\!M$, and $\LL\bar b_\alpha : \alpha < \kappa\RR \subseteq {}^\zeta\! M$ such that 
$$
M \models \varphi[\bar a_\alpha,\bar b_\beta] \wedge \neg \varphi[\bar a_\beta,\bar b_\alpha]
$$ 
whenever $\alpha < \beta < \kappa$. 

(Actually, $\kappa \ge \beth_{\gamma^*+1}
(|\tau|^{< \theta})^+$ will suffice.)

\mn
$4)$ If $p \in \bfS^\eps_\Lambda(A,M)$, 
$\varphi(\bar x,\bar y) \in \Lambda^\pe_{\gamma^*,\eps,\xi}$, and 
$p \cap \{\varphi(\bar x,\bar b) : \bar b \in {}^\xi\!A\} \ne \varnothing$, 
\underline{then} for some $\bar a_\varphi \in {}^\eps\!A$ and $\bar b \in {}^\xi\!A$ we have $\varphi(\bar x - \bar a_\varphi,\bar b) 
\vdash p \rest \{\pm \varphi\}$ and $\varphi(\bar x-\bar
a_\varphi,\bar b) \in p$.
\end{theorem}

\sn
\begin{PROOF}{\ref{b7}}
1) Consider the statement:
\begin{enumerate}
    \item[$\circledast$] If $\varphi(\bar x,\bar y) \in \Lambda^\pe_{\gamma^*,\eps,\xi} \cap \Lambda$, 
    $$
    p_\ell(\bar x) \defeq \tp_{\{\varphi(\bar x,\bar y)\}}\!(\bar a_\ell,A,M) \in \bfS^\eps_{\{\varphi(\bar x,\bar y)\}}(A,M)
    $$
    for $\ell=1,2$, $\bar b \in {}^\xi\!A$ and $\bar c \in {}^{\eps + \xi}\!A$, and $\varphi(\bar x \caret \bar b - \bar c) \in p_1(\bar x) \cap p_2(\bar x)$, \underline{then} $p_1(\bar x) = p_2(\bar x)$.
\end{enumerate}

\medskip
Why is $\circledast$ true? Assume 
$\varphi(\bar x \caret \bar b' - \bar c') \in p_1(\bar x)$, so 
$\bar a_1 \caret \bar b' - \bar c' \in \varphi(\olsi M)$. But we are assuming 
$\varphi(\bar x \caret \bar b - \bar c) \in p_\ell(\bar x) = \tp_{\{\varphi(\bar x,\bar y)\}}(\bar a_\ell,A,M)$, hence $\bar a_\ell \caret \bar b - \bar c \in \varphi(M)$ 
for $\ell=1,2$. Together, 
$$\bar a_2 \caret \bar b' - \bar c' = (\bar a_2 \caret \bar b - \bar c) - (\bar a_1 \caret \bar b - \bar c) + (\bar a_1 \caret \bar b' - \bar c') \in \varphi(M),$$ 
hence $\varphi(x \caret \bar b' - c') \in p_2(x)$. So
$\varphi(\bar x \caret \bar b' -\bar c') \in p_1 \Rightarrow \varphi(\bar x \caret b' - \bar c') \in p_2$, and by symmetry we have
`$\Leftrightarrow$,' hence $p_1(\bar x) = p_2(\bar x)$. I.e.\ we have
proved $\circledast$.

\medskip
Why is $\circledast$ sufficient? For every $\xi <\theta,\varphi(\bar x,\bar
y) \in \Lambda^\pe_{\gamma^*,\eps,\xi} \cap \Lambda$ 
and $p(\bar x) \in \bfS^\eps_\Lambda(A,M)$ 
choose $(\bar b_{p(\bar x),\varphi(\bar
x,\bar y)},\bar c_{p(\bar x),\varphi(\bar x,\bar y)})$ such that
\begin{enumerate}
    \item[$\oplus_1$] 
    \begin{enumerate}
        \item $\bar b_{p(\bar x),\varphi(\bar x,\bar y)} \in {}^\eps\!A$ and $\bar c_{p(\bar x),\varphi(\bar x,\bar y)} \in {}^{\eps + \xi} A$
\sn
        \item If possible, $\varphi(\bar x \caret \bar b_{p(\bar x),\varphi(\bar x,\bar y)} - \bar c_{p(\bar x),\varphi(\bar x,\bar y)}) \in p(\bar x)$.
    \end{enumerate}
\end{enumerate}

\mn
For $p(\bar x) \in \bfS^\eps_\Lambda(A,M)$, let
$\Phi_{p(\bar x)} \defeq \{\varphi(\bar x,\bar y) \in \Lambda^\pe_{\gamma,\eps,\xi} : \oplus_1\text{(b) does hold}\},$ and let 
$$
q_{p(\bar x)} \defeq \big\{\varphi(\bar x \caret  \bar b_{p(\bar x),\varphi(\bar x,\bar y)} - \bar c_{p(\bar x),\varphi(\bar x,\bar y)}) : \varphi(\bar x,\bar y) \in \Phi_{p(\bar x)}\big\}.
$$

\mn
Now,
\begin{enumerate}
    \item[$\oplus_2$] If $p_1(\bar x),p_2(\bar x) \in \bfS^\eps_\Lambda(A,M)$, $\Phi_{p_1(\bar x)} = \Phi_{p_2(\bar x)}$, and $q_{p_1(\bar x)} = q_{p_2(\bar x)}$, \underline{then} $p_1(\bar x) = p_2(\bar x)$.
\end{enumerate}
[Why? Just think about it.]

\mn
\begin{enumerate}
    \item[$\oplus_3$]  $\big| \big\{(\Phi_{p(\bar x)},q_{p(\bar x)}) : p(\bar x) \in \bfS^\eps_\Lambda(A,M) \big\} \big| \le 2^{|\Lambda|} + (|A|^{< \theta})^{|\Lambda|}$
\end{enumerate}
[Why? Straightforward.] 

\smallskip
Clearly we are done.

\mn
2) Note that $\varphi(\bar x,\bar y) \in \Lambda^\pe_{\gamma,\eps,\xi}$ implies that
\begin{enumerate}
    \item[$\boxplus$] $M \models ``\varphi[\bar a,\bar b] \wedge \varphi[\bar a,\bar b'] \wedge \varphi[\bar a',\bar b]" \Rightarrow M \models ``\varphi[\bar a',\bar b']"$.
\end{enumerate}
\mn
[Why? As $\varphi({}^{\eps + \zeta}\!M)$ is a subgroup of ${}^{\eps + \zeta}\!M$ and 
$\bar a \caret \bar b$, $\bar a' \caret \bar b$ and $\bar a \caret \bar b'$ belong to it. Therefore so does $\bar a' \caret \bar b + (\bar a \caret \bar b') - (\bar a \caret \bar b)$, but that is equal to $\bar a' \caret \bar b'$.]

\smallskip
So we can choose $\bar a = \bar a_0$, $\bar a' = \bar a_3$, $\bar b = \bar
b_1$, and $\bar b' = \bar b_2$, and get a contradiction.

\mn
3) Toward contradiction, let $\LL \bar a_\alpha : \alpha < \kappa\RR \subseteq {}^\eps\!M$ form a counterexample.
By the Erd\H os-Rado Theorem, 
$$
\beth_{\gamma^*+2}\big(|\tau|^{< \theta}\big)^+ \to \big(4 \big)^2_{\beth_{\gamma^*+1}(|\tau|^{<\theta})}.
$$ 
Now for $\alpha < \beta < \kappa$, let 
$p_{\alpha,\beta} \defeq \tp_{\Lambda^\pe_{\gamma^*,\eps,\eps}}\!(\bar a_\alpha \caret \bar a_\beta;\varnothing,M)$. 
So $\{p_{\alpha,\beta} : \alpha < \beta\}$ has cardinality 
$\le \beth_{\gamma^*+1}(|\tau|^{<\theta})$; 
hence by the arrow above, for some $p$ and some $\alpha_0 < \alpha_1 < \alpha_2 < \alpha_3$ we have 
$$
(\forall  k <\ell < 4)[p_{\alpha_k,\alpha_\ell} = p].
$$ 
We get a contradiction by part (2). 

If $\kappa$ is just $\ge \beth_{\gamma^*+1}(|\tau|^{< \theta})^+$, 
use clause $\boxplus$ from the proof of part (2) and 
repeat a proof of the Erd\H os-Rado Theorem.

\mn
4) Should be clear.
\end{PROOF}

\bn
Recall (from \cite{Sh:300a})
\begin{definition}\label{b8}
For $\Phi \subseteq \Lambda$, we say $\bfI \subseteq {}^\eps\!M$ is 
$(\mu,\Phi)$-\emph{convergent} \underline{when} ($|\bfI| \ge \mu$ and) for
every $\xi < \theta$, $\varphi(\bar x) \in \Phi_{\eps +\xi}$, and 
$\bar b \in {}^\xi \!M$ and $\bar c \in {}^{\xi + \eps}\!M$, for
all but $< \mu$ of the $\bar a \in \bfI$, the truth value of $\bar a \caret \bar b - \bar c \in \varphi(M)$ is constant.
\end{definition}

\bn
\begin{claim}\label{b10}
$1)$ The following is a sufficient condition for $\bfI = \{\bar a_i : i < \lambda\} \subseteq {}^\eps\!M$ to be $(\mu,\Phi)$-convergent: 
$$
i<j<\lambda \wedge \varphi(\bar x) \in \Phi \cap \Lambda_\eps \Rightarrow \bar a_j - \bar a_i \in \varphi(M).
$$

\mn
$2)$ If $\eps < \theta$, $\lambda = \cf(\lambda) > \mu \ge \mu_{\gamma^*}$, 
$(\forall i < \lambda) \big[ |i|^{\mu_{\gamma^*}} < \lambda \big]$, and 
$\LL\bar a_i : i < \lambda\RR \subseteq {}^\eps\!M$ is without repetition, \underline{then} for some stationary $S \subseteq \lambda$, $\{\bar a_i : i \in S\}$ is $(\mu^+,\Phi)$-convergent.
\end{claim}

\mn
\begin{remark}
1) Note that being $(\mu,\bfI)$-convergent is very close to being
$\lepref{\omega}$-indiscernible, and is sometimes the reasonable generalization of indiscernibility.

\mn
2) So \ref{b10}(1) says that 2-indiscernible \emph{almost} implies 
$\lepref{\omega}$-indiscernible.

\mn
3) Also, \ref{b10}(2) says that there are $\lepref{\omega}$-indiscernibles.
\end{remark}

\mn
\begin{PROOF}{\ref{b10}}
Should be clear.
\end{PROOF}

\newpage

\section {How much does the subgroup exhaust a group?}

\begin{explanation}\label{p99}
1) The motivation for this section comes from the old proof of Theorem \ref{a10}, but we find it interesting in its own right. It considers the question 
\begin{quotation}
    Given a group $G$ and a sequence $\LL G_s : s \in S\RR$ of its subgroups, under what conditions do we have $G = \bigcup\limits_s G_s$? 
\end{quotation}

\mn
2) More specifically, Definition \ref{p0} is a strong way to say $G \neq \bigcup\limits_s G_s$, and gives a more exact way to measure how different they are.

\mn
\textbf{Question:} Why?

\mn
\textbf{Answer:} Assume $S \in I_\lambda$, as witnessed by $\LL g_\alpha : \alpha < \lambda\RR$ (see \ref{p0}) and $G = \bigcup\limits_{s \in S} G_s$, and we shall prove that $\lambda \leq |S|$.

\mn
[Why? As $G  = \bigcup\limits_s G_s$ and $\{g_\alpha : \alpha < \lambda\} \subseteq G$, clearly there is a function $f : \lambda \to S$ such that $g_\alpha \in G_{f(\alpha)}$ for all $\alpha < \lambda$. Now $f$ must be one-to-one, because 
$$
\alpha \neq \beta \wedge f(\alpha) = f(\beta) \Rightarrow g_\alpha G_{f(\alpha)} = G_{f(\alpha)} = G_{f(\beta)} = g_\beta G_{f(\beta)} = g_\beta G_{f(\alpha)},
$$
contradicting the definition of $\bar g$.]

\mn
3) Recall that in the old proof of \ref{a10} we used only \ref{p0} and \ref{p2}$(a)$-$(b)$, and from \ref{p1}(3) we use only the sentence starting with `Moreover.'

Clauses \ref{p2}$(c)$-$(g)$ are not used elsewhere, but we still find them interesting: in particular clause $(d)$, giving sufficient conditions for $I_\lambda$ to be an ideal.

\mn
4) An obvious point --- {but still} we note that $\lambda$ is free to vary, so $I_{G,\olsi G,\lambda}$ is defined for all $\lambda$, even with $G,\olsi G$ fixed.


\mn
5) In this section, we may even allow the $G_s$ to be affine subsets of $G$ in the Abelian case. (So $G_s = g_sG_s'$ for some $g_s \in G$, where $G_s'$ is a subgroup of $G$.)

This makes no difference.
\end{explanation}

\bn
\begin{definition} \label{p0}
Assume $G$ is a group and 
$\olsi G = \LL G_s : s \in S \RR $ 
is a sequence of subgroups of $G$. 

\mn
1) For $\lambda \ge 0,$ let 
$I = I_\lambda = I_{G, \olsi G, \lambda}$ be the set of 
$u \subseteq S$ which are \emph{witnessed} 
by some sequence $\bar g = \LL g_\alpha : \alpha < \lambda\RR \subseteq G$.

By this we mean
$$
s \in u \wedge \alpha < \beta < \lambda\ \Rightarrow\ g_\alpha G_s \ne g_\beta G_s.
$$

\mn
2) For $ \lambda \ge  0 
$ let $ I_{<\lambda} = I_{G, \olsi G, <\lambda} \defeq  \bigcap \limits_{\mu<\lambda} I_\mu$.

\mn
3) Let $I^+ \defeq \clP(S) \setminus I$ for any $ I \subseteq \clP(S)$.  
\end{definition}

\bn
\begin{claim}\label{p1} 
Let $ G, S, \olsi G $ be as in Definition \emph{\ref{p0}}; hence $ I_\mu = I_{G, \olsi G,\mu}$ is well-defined for any cardinal $\mu$.

\mn
$1)$ For any $\lambda$, the sequence  $\LL I_\mu : \mu < \lambda \RR $ is $\subseteq$-decreasing
(that is, $ \mu < \lambda \Rightarrow I_\lambda \subseteq I_\mu$).

\mn
$2)$ If in addition $\lambda$ has cofinality $ > 2^{|S|}$, \underline{then} the sequence 
$\LL I_\mu : \mu < \lambda  \RR$ is eventually constant. 

\mn
$3)$ There is $\xi < (2^{|S|})^{+4}$ of cofinality 
$(2^{|S|})^+$ such that $\lambda \defeq \aleph_\xi$ 
satisfies all the demands mentioned in \emph{\ref{p2}} below. 

Moreover, if (e.g.) $ \lambda = (2^{|S|})^+$, \underline{then} $\lambda$ satisfies the demands 
in \emph{\ref{p2}(a),(b)}.
 
\mn
$4)$ Similarly, {when} $\lambda \defeq \beth_\delta$ with $\delta = (2^{|S|})^+$, or just $\cf(\delta) > 2^{|S|}$.

\mn
5) If $\lambda > |G|^{<\theta}$ then $I_{G,\olsi G,\lambda} = \{\varnothing\}$.
\end{claim}

\begin{PROOF}{\ref{p1}}
1) Obvious.

\mn
2) If $\lambda$ is a successor cardinal this is trivial. 

If $\lambda$ is limit, then $\LL I_\mu : \mu < \lambda\RR$ is a $\subseteq$-decreasing sequence of subsets of $\clP(S)$, which has cardinality $2^{|S|}$. The conclusion is clear.

\mn
3) The second sentence is obvious, but for the first sentence we will need to quote results from elsewhere.
Recall\footnote{
    $\pp(\mu)$ is defined for singular cardinals $\mu$ in \cite[Ch.II,\,1.1]{Sh:g} (or in \cite[6.1]{Sh:E12}), but this is not used {in this section}.
}
\begin{enumerate}
    \item [$(*)_1$]
    If $\kappa$ is regular, $\mu > \cf(\mu) \geq \kappa$, and 
    $$
    (\forall \chi < \mu)\big[ \cf(\chi) < \kappa \Rightarrow\pp(\chi) < \kappa \big],
    $$
    then $\cf([\mu]^{< \kappa},\subseteq) = \mu$.
\end{enumerate}
[Why? By \cite[Ch.II,\,5.1,5.4,6.12]{Sh:g} (or \cite[5.1,7.1,7.4]{Sh:E12}).]

\begin{enumerate}
    \item [$(*)_2$] $\cf(\mu) < \mu \Rightarrow \pp(\mu) \leq \mu^{\cf(\mu)}$.
\end{enumerate}
[Why? Obvious from the definition of $\pp(-)$.]

\begin{enumerate}
    \item [$(*)_3$] Assume $\kappa = \aleph_{\xi+\zeta}$ is singular, and $\zeta < \aleph_\zeta$. 
    If $|\zeta| = \kappa^{++}$ and $\cf(\zeta) < \kappa = \cf(\kappa)$, then $\pp(\aleph_{\xi+\zeta}) < \aleph_{\xi+|\zeta|^+}$.
\end{enumerate}
[Why? By \cite[\S4]{Sh:400}.]

\begin{enumerate}
    \item [$(*)_4$] If $2^\kappa \leq \mu$ and $\cf([\mu]^\kappa,\subseteq) = \mu$ then $\mu^\kappa = \mu$.
\end{enumerate}
[Why? Obvious.]

\begin{enumerate}
    \item [$(*)_5$] If $\theta \leq \mu = \mu^\theta$ then for some club $E$ of $\theta^{+4}$ we have
    \begin{enumerate}
        \item If $\delta \in E$ with $\cf(\delta) > \theta$ and $\xi \defeq \min(E \setminus \delta+1)$, then $\pp(\mu^{+\delta}) < \mu^{+\xi}$.
\sn
        \item If $\delta \in \acc(E) \defeq \big\{ \alpha \in E : \alpha = \sup(E\cap \alpha) \big\}$ and $\cf(\delta) > \theta$, \underline{then} $\cf\big( [\mu^{+\delta}]^{\leq\theta},{\subseteq}\big) = \mu^{+\delta}$.
\sn
        \item Moreover, in part (b) we have $(\mu^{+\delta})^\theta = \mu^{+\delta}$.
    \end{enumerate}
\end{enumerate}
[Why? By the previous statements.]

Now part (3) is a special case of $(*)_5$
as $(2^{|S|})^{|S|} = 2^{|S|}$.

\mn
4) Easy as well.

\mn
5) If $u \in I$ as witnessed by $\LL g_\alpha : \alpha < \lambda\RR$ then $\LL g_\alpha G_s : \cyan{\alpha < \lambda}\RR$ is without repetition, hence $\LL g_\alpha : \alpha < \lambda\RR$ is without repetition. Hence $|G| \geq \lambda$ --- a contradiction.
\end{PROOF}

\bn
\begin{claim}\label{p2}
Assume $\LL G_s : s \in S\RR$ is a sequence of subgroups of the 
group $G$  
  and $ \lambda $ a cardinal. 
The set $ I = I_\lambda = I_{G, \olsi G, \lambda}$ satisfies:
\mn
\begin{enumerate}[$(a)$]
    \item If $S \notin I$, $\cf(\lambda) > 2^{|S|}$, and 
    $\alpha < \lambda \Rightarrow |\alpha|^{|S|} < \lambda$ 
    (e.g.\ $(\exists \mu)[\lambda = (\mu^{|S|})^+]$), \underline{then} there is 
    $A \subseteq G$ of cardinality $< \lambda$ such that 
    $$
    (\forall g \in G)( \exists a \in A)\big[\{s \in S : g G_s \ne a G_s\} \in I \big].
    $$
    \item If $S \notin I$,  
    $\alpha < \lambda \Rightarrow |\alpha|^{|S|} < \lambda$, and $\lambda$ is regular,
    \underline{then} 
    for every $u \in I$ there exists $\bar g$ and $v \in I$ such that
    \begin{enumerate}
        \item[$\bullet_1$] $u \subseteq v$
\sn
        \item[$\bullet_2$] $\bar g = \LL g_\alpha : \alpha < \lambda \RR$
\sn
        \item[$\bullet_3$] $g_\alpha G_s = g_0 G_s$ for all $\alpha < \lambda$ and $s \in S \setminus v$. Moreover, 
        $$
        \alpha < \lambda \Rightarrow g_\alpha \in \bigcap\limits_{s \in S \setminus v} G_s.
        $$
        \item[$\bullet_4$] If $s \in v$ and $0< \alpha < \beta < \lambda$,
        then $g_\alpha G_s \ne g_\beta G_s$.
    \end{enumerate}
\sn
    \item $I \subseteq \clP(S)$ is closed under subsets.
\sn
    \item $I = I_\lambda$ is an ideal, \emph{provided that}
    \begin{enumerate}[$\bullet_1$] 
        \item For some $\theta \in \big(|S|, \lambda \big)$ we have $ I_\theta = I_\lambda$.
\sn
        \item $G$ is Abelian (\underline{or} just each $G_s$ is a normal subgroup of $G$).
    \end{enumerate} 
\sn
    \item Assuming clause $(d)\bullet_2$ and $\lambda > |S|^+$,
    $$
    (\forall u_1,u_2 \in I_\lambda)( \forall \mu < \lambda)[ u_1 \cup u_2 \in I_\mu].
    $$
    \item Assuming clause $(d)\bullet_2$, if $|S| < \aleph_0$ then $I$ is an ideal.
    \item The following holds:
\begin{enumerate}[$\bullet_1$] 
    \item If $ u \in I_\lambda $ and $\{g_\alpha : \alpha < \alpha_*\}$ 
    is $\subseteq$-maximal such that 
    $$
    \alpha < \beta < \alpha_* \wedge s \in S \Rightarrow g_\alpha G_s \neq g_\beta G_s,
    $$  
    then $ \alpha_* \in [\lambda, \lambda^+)$.
\sn
    \item If $ \lambda $ is a limit cardinal of cofinality 
         $> 2^{|S|}$ then $(\exists \theta < \lambda)[ I_\theta = I_\lambda]$.
 \end{enumerate}    
\end{enumerate}     
\end{claim}



\mn
\begin{PROOF}{\ref{p2}}
Let $I= I_\lambda $   
be defined as in \ref{p0}.

\mn
Now,
\begin{enumerate}
    \item[$\circledast_0$] $I \subseteq \clP(S)$ is $\subseteq$-downward closed: i.e.\ $u \in I \wedge v \subseteq u \Rightarrow v \in I$.
\end{enumerate}
[Why? Obvious.]

\medskip
This covers clause $(c)$.

\medskip
Toward proving clause $(a)$ of the claim, for each 
$u \in I^+ \defeq \clP(S) \setminus I$,
let $\bar g_u = \LL g_{u,\alpha} : \alpha < \alpha_u\RR$ be a
maximal sequence of members of $G$ such that 
$$
\alpha < \beta < \alpha_u \wedge s \in u\ \Rightarrow\ g_{u,\alpha} G_s \ne g_{u,\beta} G_s.
$$
As $u \in \clP(S) \setminus I$, necessarily
$\alpha_u < \lambda$ (by the definition of $I$), and as we are assuming 
$\cf(\lambda) > 2^{|S|}$, clearly $\alpha_* \defeq \sup\{\alpha_u : u \in I^+\} < \lambda$. 
So 
$$
B \defeq \{g_{u,\alpha} : u \in I^+,\ \alpha < \alpha_u\}
$$ 
is a subset of $G$ of cardinality $< \lambda$. 

Next,
\begin{enumerate}
    \item [$\circledast_1$] For every $u \in I$ and $h:S \setminus u \to B$, choose $g_h \in G$ such that 
    $$
    \boxplus_{h,g_h} \vee(\forall g \in G)\big[\neg\boxplus_{h,g} \big],
    $$
    where 
\sn
    \item [$\boxplus_{h,g}$]
    \begin{itemize}
        \item $g \in G$
\sn
        \item $h:S \setminus u \to B$
\sn
        \item $(\forall s \in S \setminus u)[g G_s = h(s) G_s]$.
    \end{itemize}
\end{enumerate}
Explicitly, if there exists such a $g$ then {choose one of them as our $g_h$}; otherwise let $g_h \defeq g_{u,0}$, just so that it is defined.

Now
$$
A \defeq \{g_h : h \in {}^{S \setminus u}\!B,\ u \in I, \text{ and } \boxplus_{h,g_h}\}
$$ 
is a subset of $G$ of cardinality\footnote{
    Recall that we are assuming $(\forall \alpha < \lambda) \big[|\alpha|^{|S|} < \lambda \big]$.
} 
$\le |B|^{|S|} < \lambda$.  

For showing $A$ is as required in clause $(a)$, fix $g_* \in G$. Let 
$$
v \defeq \big\{s \in S : (\forall w \in I^+)(\forall\alpha < \alpha_w) [g_* G_s \neq g_{w,\alpha} G_s] \big\}.
$$
Now if $v \in I^+$ then $\bar g_v = \LL g_{v,\alpha} : \alpha < \alpha_v\RR$ is well-defined
and $g_*$ satisfies\\ 
$(\forall\alpha < \alpha_v) [g_* G_s \neq g_{v,\alpha} G_s]$,
contradicting the maximality of $\bar g_v$. 

Therefore $v \in I$ by our choices (as $I \cup I^+ = \clP(S)$ and $v \subseteq S$), so by the definition of $v$ we can find a function $h : S \setminus v \to B$ such that 
$$
s \in S \setminus v \Rightarrow g_* G_s = h(s) G_s.
$$ 
So $g_*$ and $h$ satisfy $\boxplus_{h,g_*}$, hence by $\circledast_1$ there is a $g_h \in A$ satisfying $\boxplus_{h,g_h}$, so clause $(a)$ holds with $a \defeq g_h$.

\bigskip
\centerline{*\qquad*\qquad*}

\medskip 
For clause $(b)$, let $u \in I$ be given and let $\LL g_\alpha : \alpha < \lambda\RR$ witness that $u \in I$. For each $\alpha < \lambda$, let 
$$
u_\alpha \defeq \big\{s \in S : (\exists\beta < \alpha ) [g_\alpha G_s = g_\beta G_s] \big\}.
$$ 
Clearly $u_\alpha \cap u = \varnothing$; let $h_\alpha : u_\alpha \to \alpha$ be such that
$s \in u_\alpha \Rightarrow g_\alpha G_s = g_{h_\alpha\!(s)}G_s$.

As $\lambda$ is regular 
and   
recalling 
$(\forall \alpha < \lambda) \big[ |\alpha|^{|S|} < \lambda \big]$ 
by the present assumptions,   
for some $u_* \subseteq S$ and $h: u_* \to \lambda$, the set 
$$
\clW \defeq \{\alpha < \lambda : \cf(\alpha) = |S|^++\aleph_0,\ h_\alpha = h,\ u_\alpha = u_*\}
$$
is a stationary subset of $\lambda$.
Clearly 
$$
\alpha,\beta \in \clW \wedge s \in u_* \Rightarrow g_\alpha G_s = g_{h(s)} G_s = g_\beta G_s
$$ 
and 
$$
\alpha, \beta \in \clW \wedge \alpha \ne \beta \wedge s \in S \setminus u_* \Rightarrow g_\alpha G_s \ne g_\beta G_s.
$$
Letting $\LL \alpha_i : i < \lambda\RR$ list $\clW$ in increasing order and letting $g_i' \defeq g_{\alpha_i}$ for $i < \lambda$, clearly
$v \defeq u_*$ and $\LL g'_i : i < \lambda\RR$ 
are as promised in clause $(b)$.

For the `moreover' bit, 
recalling clause $(b)$ just says ``for some $\bar g$ and $v$," we let $g_\alpha'' \defeq (g_0')^{-1}g_\alpha$ and use the sequence $\LL g_\alpha'' : \alpha < \lambda\RR$. That is, 
first

$$
\sigma < \beta < \lambda \wedge s \in S \Rightarrow 
g_\alpha G_s \not= g_\beta G_s \Rightarrow 
(g_0)^{-1} (g_\alpha G_s )\not= (g_0)^{-1}  (g_\alpha G_s ) \Rightarrow 
g_\alpha '' G_s \not= g_\beta '' 
$$

Second 

$$
\alpha < \lambda \wedge s \in S \setminus v \Rightarrow g_\alpha''G_s = (g_0')^{-1}(g_\alpha G_s) = 
(g_0')^{-1}(g_0' G_s) = 
\big( (g_0')^{-1}g_0' \big) G_s = G_s.
$$
Therefore (as $v \defeq S \setminus u_*$) we have
$$
\alpha < \lambda \wedge s \in v\  \Rightarrow\ g_\alpha''G_s = G_s\ \Rightarrow\ g_\alpha'' \in G_s
$$
as promised.

\bigskip
\centerline{*\qquad*\qquad*}

\bigskip
Lastly, it just remains to prove clause $(e)$, as $(d)$ is an immediate consequence and clause $(f)$ is easier 
and clause $(g)$ is proved as in clause $(a)$.  

Let $u_1,u_2 \in I_\lambda$ be disjoint, and we shall prove that $u \defeq  u_1 \cup u_2 \in I_\mu$ when $\mu \in \big(|S|, \lambda \big)$. 
Let $\LL g_{\ell,\alpha} : \alpha < \lambda\RR$ witness `$u_\ell \in I_\lambda$' for $\ell=1,2$. 

We try to choose $g_{3,\eps} \in G$ by induction on $\eps < \mu$ such that 
$$
\zeta < \eps \wedge s \in u\ \Rightarrow\ g_{3,\eps} G_s \ne g_{3,\zeta} G_s;
$$
we shall also demand that $g_{3,\eps} \in \{g_{1,i} g_{2,j} : i,j < \lambda\}$.

Arriving to $\eps$, if for some $i , j < \lambda$ we can choose 
$g_{3,\eps} \defeq g_{1,i} g_{2,j}$, then we are done. 

Towards contradiction, assume there are no such $i$ and $j$. Then we have\\
$f : \lambda \times \lambda \to \eps$ and $h : \lambda \times \lambda \to u$ such that for every 
$(i,j) \in \lambda \times \lambda$ we have 
$$
g_{1,i} g_{2,j} G_{h(i,j)} = g_{3,f(i,j)} G_{h(i,j)}.
$$ 

For each $i< \lambda$, $\zeta < \eps$, and $s \in u \subseteq S$, let 
$$
\clU^2_{i,\zeta,s} \defeq \{j < \lambda : f(i,j) = \zeta,\ h(i,j) = s\}.
$$
Now $j \in \clU^2_{i,\zeta,s}\ \Rightarrow\ g_{1,i} g_{2,j} G_s = g_{3,\zeta} G_s\ \Rightarrow\ g_{2,j} G_s = g^{-1}_{1,i} g_{3,\zeta} G_s$; hence if $s \in u_2$ then 
$$
j \ne k \in \clU^2_{i,\zeta,s} \Rightarrow g_{2,j} G_s = (g^{-1}_{1,i} g_{3,\zeta}) G_s = g_{2,k} G_s
$$ 
--- the last statement contradicts our choice of $\LL g_{2,j} : j < \lambda\RR$. Hence $\clU^2_{i,\zeta,s}$ has cardinality $\le 1$ 
for all $i < \lambda$, $\zeta < \eps$, and $s \in u_2$.

For $j < \lambda$, $\zeta < \eps$, and $s \in u$, let
\[
\clU^1_{j,\zeta,s} \defeq \big\{i < \mu : f(i,j) = \zeta \text{ and } h(i,j) = s \big\}.
\]

If $G$ is Abelian, then (as above) we have $\zeta < \eps \wedge j < \lambda \wedge s \in u_1 \Rightarrow |\clU^1_{j,\zeta,s}| \le 1$. 
If $G$ is non-Abelian and every $G_s$ is a normal subgroup of $G$, then for any
$j < \lambda$, $\zeta < \mu$, $s \in u_1$ we have 
\begin{align*}
    i \in \clU^1_{j,\zeta,s}\ &\Rightarrow\ g_{1,i} g_{2,j} G_s = g_{3,\zeta} G_s\\ &\Rightarrow\ g_{1,i}(G_s g_{2,j}) = g_{1,i}(g_{2,j} G_s) = g_{3,\zeta} G_s\\
    &\Rightarrow\ g_{1,i} G_s = g_{3,\zeta}(G_s g^{-1}_{2,j}).
\end{align*}
Hence $i \ne k \in \clU^1_{j,\zeta,s} \Rightarrow g_{1,i} G_s = g_{3,\zeta}(G_s g^{-1}_{2,j}) = g_{1,k} G_s$. This is a contradiction, so again $\clU^1_{j,\zeta,s}$ has at most one member.
 
For $\ell \in \{1,2\}$ and $i < \lambda$, let $\clU^\ell_i \defeq \bigcup\limits_{\zeta < \eps}  \bigcup\limits_{s \in u_\ell} \clU^\ell_{i,\zeta,s}$,
so as $|u_\ell| \le |S|$ clearly $|\clU^\ell_i| \le |S| + |\eps|$. 
Recall that {we have} $\lambda > \mu > |S| + |\eps|$, {so} there are 
$i,j < \lambda$ such that $i \notin \clU^1_j \wedge j \notin \clU^2_i$; hence the member $g_{1,i} g_{2,j}$ of $G$ 
satisfies the demand on $g_{3,\eps}$.

So we can carry the induction on $\eps < \mu$, so we are
done proving clause {$(e)$}.
\end{PROOF}

\bn
\begin{claim}\label{p5}
In Claim \emph{\ref{p2}} there is a $W \subseteq S$ such that
\begin{enumerate}
    \item[$(a)$] There is a sequence $\bar s = \LL s_i : i < i_*\RR$ listing $W$ such that $\big(\bigcap\limits_{i<j} G_{s_i}:\bigcap\limits_{i \leq j} G_{s_i}\big)$ is finite for $j < i_*$ (stipulating $\bigcap\limits_{i<0} G_{s_i} \defeq G$).
\sn
    \item[$(b)$] $W$ is $\subseteq$-maximal among all subsets of $S$ satisfying clause $(a)$ above.
\end{enumerate}
\end{claim}

\begin{PROOF}{\ref{p5}}
Immediate.
\end{PROOF}

\bn
\centerline{*\qquad*\qquad*}

\bn
We may take a more general perspective.
\begin{definition}\label{p17}
1) We say $\bfm$ is a \emph{movement system} when it consists of
\begin{enumerate}[(a)]
    \item A set $S$.
\sn
    \item A sequence of sets $\LL\clS_s : s \in S\RR$.
\sn
    \item A set $G$.
\sn
    \item Mappings $\inv_{\!g} \in \prod\limits_s \clS_s$ for each $g \in G$.
\end{enumerate}

\mn
2) For $\lambda$ a cardinal, let $I_\lambda^\bfm$ be the family of sets $u \subseteq S$ which are \emph{witnessed} by some $\bar g = \LL g_\alpha : \alpha < \lambda\RR$.
By this we mean
\begin{enumerate}[(a)]
    \item $\{ g_\alpha : \alpha < \lambda\} \subseteq G_\bfm$
\sn
    \item For each $s \in u$, the sequence $\LL\inv_{\!g_\alpha}^\bfm(s) : \alpha < \lambda\RR$ is without repetition.
\end{enumerate}
\end{definition}

\bn
\begin{claim}\label{p20}
The parallel of \emph{\ref{p2}} holds for this definition.
\end{claim}

\newpage

\section {Concluding Remarks}

\begin{example}\label{p9}
An example of an additive structure is a ring satisfying $xy = -yx$.

E.g.\ if 
$(R,+^R)$ is $\bigoplus\{\bbZ x_s : s \in I\}$, $f$ is a function from
$I \times I$ into $R$ such that $f(x,y) = -f(y,x)$ (so $f(x,x)=0$),
and we have 
\[
\Big( \sum\limits_{\ell < \ell_*} a_\ell x_{s_\ell} \Big) \,
\Big( \sum\limits_{m < n_*} b_m x_{t_n} \Big) = \sum\limits_{\ell < \ell_*} \sum\limits_{m < m_*} a_\ell b_m x_{f(s_\ell,t_m)}.
\]
\end{example}

\bn
\begin{remark}\label{z15}
1) We may use $\tau \supseteq\{+,-,0,1\} \cup \{P_i : i < i_*\}$ with $P_i$ unary, and instead of modules we use $\tau$-models $M$ such that $|M|$ is
the disjoint union $\bigcup\limits_{i < i_*} P^M_i$, $+^M$ is a 
partial two-place function which can be decomposed into
$$
+^M \defeq \bigcup\{+^M \rest P^M_i : i < i_*\},
$$
$(P^M_i,+^M)$ an Abelian group, and all relations and functions commute with $+$ (or at least every relation is \emph{affine}). 

I.e., let $F_*(x,y,z) \defeq x-y+z$ and demand 
$$
G(\ldots,F_*(x_i,y_i,z_i),\ldots)_{i < i_*} = F_*\big(G(\bar x),G(\bar y),G(\bar z)\big),
$$ 
where $F_*(\bar a,\bar b,\bar c) \defeq \big\LL F_*(a_i,b_i,c_i) : i < \arity(P) \big\RR \in P^M$ for $\bar a,\bar b,\bar c \in P^M$.

\mn
2) However, as we use infinitary logics, if $M$ is the disjoint union of Abelian groups $G^M_i \defeq (P^M_i,+^M_i)$ for $i < i_*$ and we define $G_M$ as the direct sum having predicates for those subgroups, \underline{then} we have bi-interpretability. 
When we have only ``affine structure," we can expand by 
choosing an element in each {summand} to serve as zero.

\mn
3) It is natural to extend our logic by cardinality quantifiers which say
``the definable group $G$ quotient the definable subgroup $H$ has 
cardinality $\ge \lambda$." 
\end{remark}

\bn
\begin{remark}\label{y3}
Concerning Theorem \ref{a10}:

\mn
1) Note that instead of an $R$-module $M$ we can use $(M,c_\alpha)_{\alpha < \kappa}$: 
i.e., expand $M$ by $\kappa$-many individual constants. The only difference is that we will use $\beth_\alpha(|R|^{< \theta} + \kappa)$ instead of $\beth_\alpha(|R|^{< \theta})$.

\mn
2) Theorem \ref{a10} has an arbitrary choice --- the construction of the $\bfI_\alpha$-s. 
Instead of using extra individual constants, in the proof,\footnote{
    See $\boxplus_{\beta+1}$ in the proof of \ref{a10}.
} 
for any $\psi(\bar x)$, 
$\psi(\bar x) \wedge \varphi_i(\bar x)$ for $i < i_* < \kappa_\beta$, $I$, $G$, and $\LL G_i : i < i_*\RR$, we expand $M$ by:
\begin{enumerate}[(a)]
    \item $P^M \defeq \big\{\bar a : M \models \psi[\bar a]$ and $\{i < \kappa_\beta : \bar a \notin G_i\} \in I \big\}$, which is a subgroup.
\sn
    \item Predicates for the set $\big\{\bar a + P^M : \bar a \in \psi(M) \big\}$.
\end{enumerate}
So the proof shows that in $M$ we can eliminate quantifiers to
quantifier-free formulas in this expansion.

\mn
3) Also, this may give too much information. The result gives elimination of quantifiers, but unlike the first-order case we use more than just the positive existential formulas.

\mn
4) We can now define non-forking: hopefully \cite{Sh:F1210} will deal with this.
\end{remark}

\bn
\begin{question}\label{y8}
1) Are there arbitrarily large Abelian groups $G$ which are not only indecomposable, 
but even \emph{potentially} so? (I.e.\ absolutely --- even after any forcing $G$ is indecomposable?)

\mn
2) Relatives of this --- e.g., no \emph{potential} non-trivial automorphism.
\end{question}

\bn
\begin{discussion}\label{y12}
We know that up to the minimal cardinal $\lambda$ satisfying 
$$
\lambda \to (\omega)^{< \omega}_{\aleph_0},
$$
the answer is yes (and more). But if $|G| \ge \lambda$ \underline{then} absolutely it has non-trivial endomorphisms and even non-trivial embeddings of 
$G$ into itself (Eklof-Shelah \cite{Sh:678}, G\"obel-Shelah \cite{Sh:880}). We
can improve this to ``for some $a_1 \ne a_2$ from $G$;" potentially there
are embeddings $f_1,f_2$ of $G$ into itself such that $f_1(a_1) = a_2$ and $f_2(a_2) = a_1$ --- see \cite{Sh:F1210}.
\end{discussion}
\newpage


\bibliographystyle{amsalpha}
\bibliography{shlhetal}

\end{document}

%% file: mgrimes.tex
\usepackage{xcolor}

\newcommand{\cyan}[1]{\textcolor{cyan}{#1}}

\usepackage[shortlabels]{enumitem}
\setlist[enumerate,1]{label={(\Alph*)}}
\setlist[enumerate,2]{label={(\alph*)}}
\setlist[enumerate,3]{label={$\bullet_{\arabic*}$}}

\newenvironment{PROOF}[2][\proofname.]
   {\begin{proof}[#1]\renewcommand{\qedsymbol}{\textsquare$_{\mathrm{#2}}$}}
   {\end{proof}}

\newtheorem{theorem}{Theorem}[section]

\newtheorem{claim}[theorem]{Claim}
\newtheorem{conclusion}[theorem]{Conclusion}

\newtheorem{observation}[theorem]{Observation}

\theoremstyle{definition}

\newtheorem{definition}[theorem]{Definition}
\newtheorem{discussion}[theorem]{Discussion}
\newtheorem{example}[theorem]{Example}
\newtheorem{explanation}[theorem]{Explanation}

\theoremstyle{remark}

\newtheorem{context}[theorem]{Context}
\newtheorem{notation}[theorem]{Notation}
\newtheorem{question}[theorem]{Question}
\newtheorem{remark}[theorem]{Remark}

\usepackage{graphicx}

\newcommand{\af}{\mathrm{af}}
\newcommand{\inv}{\mathrm{inv}}

\newcommand{\acc}{\mathrm{acc}}

\newcommand{\cf}{\mathrm{cf}}

\newcommand{\dom}{\mathrm{dom}}

\newcommand{\pp}{\mathrm{pp}}

\newcommand{\Th}{\mathrm{Th}}

\newcommand{\tp}{\mathrm{tp}}

\newcommand{\Mod}{\mathrm{Mod}}

\newcommand{\pe}{\mathrm{pe}}

\newcommand{\arity}{\mathrm{arity}}

\newcommand{\bfI}{\mathbf{I}}

\newcommand{\bfK}{\mathbf{K}}

\newcommand{\bfS}{\mathbf{S}}

\newcommand{\bfe}{\mathbf{e}}
\newcommand{\bff}{\mathbf{f}}

\newcommand{\bfm}{\mathbf{m}}

\newcommand{\bfs}{\mathbf{s}}

\newcommand{\bbL}{\mathbb{L}}

\newcommand{\bbZ}{\mathbb{Z}}

\newcommand{\mn}{\medskip\noindent}
\newcommand{\sn}{\smallskip\noindent}
\newcommand{\bn}{\bigskip\noindent}

\newcommand{\clI}{\mathcal{I}}

\newcommand{\clP}{\mathcal{P}}

\newcommand{\clS}{\mathcal{S}}

\newcommand{\clU}{\mathcal{U}}

\newcommand{\clW}{\mathcal{W}}

\newcommand{\gd}{\mathfrak{d}}

\newcommand{\eps}{\varepsilon}
\newcommand{\cl}{c\kern-.11ex \ell}
\newcommand{\lh}{{\ell\kern-.27ex g}}
\newcommand{\rest}{\restriction}

\newcommand{\caret}{{\char 94}}
\newcommand{\LL}{\langle}
\newcommand{\RR}{\rangle}

\newcommand{\lepref}[1]{({<}\,{#1})}

\usepackage{tcolorbox}

\newcommand{\olsi}[1]{\,\overline{\!{#1}}} 

\newcommand*{\defeq}{\mathrel{\vcenter{\baselineskip0.5ex \lineskiplimit0pt\hbox{\scriptsize.}\hbox{\scriptsize.}}}=}

\newcommand{\leqdot}{\mathrel{\mathord{\leq}\!\!\raise
0.8 pt\hbox{$\cdot$}}}

\newcount\skewfactor
\def\mathunderaccent#1#2 {\let\theaccent#1\skewfactor#2
\mathpalette\putaccentunder}
\def\putaccentunder#1#2{\oalign{$#1#2$\crcr\hidewidth
\vbox to.2ex{\hbox{$#1\skew\skewfactor\theaccent{}$}\vss}\hidewidth}}

\newbox\noforkbox \newdimen\forklinewidth
\setbox0\hbox{$\textstyle\bigcup$}
\setbox1\hbox to \wd0{\hfil\vrule width \forklinewidth depth \dp0
                        height \ht0 \hfil}
\setbox\noforkbox\hbox{\box1\box0\relax}
\def\unionstick{\mathop{\copy\noforkbox}\limits}
\def\nonfork#1#2_#3{#1\unionstick_{\textstyle #3}#2}
\def\nonforkin#1#2_#3^#4{#1\unionstick_{\textstyle #3}^{\textstyle
    #4}#2}
\setbox0\hbox{$\textstyle\bigcup$}
\setbox1\hbox to \wd0{\hfil{\sl /\/}\hfil}
\setbox2\hbox to \wd0{\hfil\vrule height \ht0 depth \dp0 width
                                \forklinewidth\hfil}
\newbox\doesforkbox
\setbox\doesforkbox\hbox{\box1\box0\relax}
\def\nunionstick{\mathop{\copy\doesforkbox}\limits}

\def\fork#1#2_#3{#1\nunionstick_{\textstyle #3}#2}
\def\forkin#1#2_#3^#4{#1\nunionstick_{\textstyle #3}^{\textstyle
    #4}#2}

\newcommand{\stickT}{%
\setbox255=\hbox{\raise1ex\hbox{$\hspace{0.2pt}\,\bullet\,$}}
\mathord{\rlap{\hbox to\wd255{\hss\hbox{$|$}\hss}}
\box255}
}
\newcommand{\stickS}{%
\setbox255=\hbox{\raise0.6ex\hbox{$\scriptstyle\bullet$}}
\mathord{\rlap{\hbox to\wd255{\hss\hbox{$\scriptstyle|$}\hss}}
\box255}
}

%% file: 977.bbl
\newcommand{\etalchar}[1]{$^{#1}$}
\providecommand{\bysame}{\leavevmode\hbox to3em{\hrulefill}\thinspace}
\providecommand{\MR}{\relax\ifhmode\unskip\space\fi MR }
\providecommand{\MRhref}[2]{%
  \href{http://www.ams.org/mathscinet-getitem?mr=#1}{#2}
}
\providecommand{\href}[2]{#2}
\begin{thebibliography}{She94b}

\bibitem[AGS25]{Sh:1246}
Mohsen Asgharzadeh, Mohammad Golshani, and Saharon Shelah, \emph{{Expressive
  power of infinitary logic and absolute co-Hopfianity}}, Illinois J. Math
  \textbf{69} (2025), no.~2, 269--302,
  \href{https://arxiv.org/abs/2309.16997}{arXiv: 2309.16997}. \MR{4919937}

\bibitem[Bau76]{Bau76}
Walter Baur, \emph{Elimination of quantifiers for modules}, Israel J. Math.
  \textbf{25} (1976), 64--70.

\bibitem[Ekl71]{Ek71}
Paul~C. Eklof, \emph{Homogeneous universal modules}, Math. Scand. \textbf{29}
  (1971), 187--196.

\bibitem[ES99]{Sh:678}
Paul~C. Eklof and Saharon Shelah, \emph{{Absolutely rigid systems and
  absolutely indecomposable groups}}, {Abelian groups and modules (Dublin,
  1998)}, Trends Math., Birkh\"auser, Basel, 1999,
  \href{https://arxiv.org/abs/math/0010264}{arXiv: math/0010264}, pp.~257--268.
  \MR{1735574}

\bibitem[Fis77]{Fis77}
Edward~R. Fisher, \emph{Abelian structures. i.}, Abelian group theory (Proc.
  Second New Mexico State Univ. Conf., Las Cruces, N.M., 1976) (Berlin),
  Lecture Notes in Math., vol. 616, Springer, 1977, pp.~270--322.

\bibitem[GS07]{Sh:880}
R{\"u}diger G{\"o}bel and Saharon Shelah, \emph{{Absolutely indecomposable
  modules}}, Proc. Amer. Math. Soc. \textbf{135} (2007), no.~6, 1641--1649,
  \href{https://arxiv.org/abs/0711.3011}{arXiv: 0711.3011}. \MR{2286071}

\bibitem[S{\etalchar{+}}]{Sh:F1210}
S.~Shelah et~al., \emph{Tba}, In preparation. Preliminary number: Sh:F1210.

\bibitem[She]{Sh:E12}
Saharon Shelah, \emph{{Analytical Guide and Updates to [Sh:g]}},
  \href{https://arxiv.org/abs/math/9906022}{arXiv: math/9906022} Correction of
  [Sh:g].

\bibitem[She71]{Sh:11}
\bysame, \emph{{On the number of non-almost isomorphic models of $T$ in a
  power}}, Pacific J. Math. \textbf{36} (1971), 811--818. \MR{0285375}

\bibitem[She90]{Sh:c}
\bysame, \emph{{Classification theory and the number of nonisomorphic models}},
  2nd ed., Studies in Logic and the Foundations of Mathematics, vol.~92,
  North-Holland Publishing Co., Amsterdam, 1990, Revised edition of [Sh:a].
  \MR{1083551}

\bibitem[She94a]{Sh:g}
\bysame, \emph{{Cardinal arithmetic}}, Oxford Logic Guides, vol.~29, The
  Clarendon Press, Oxford University Press, New York, 1994. \MR{1318912}

\bibitem[She94b]{Sh:400}
\bysame, \emph{{Cardinal Arithmetic}}, {Cardinal Arithmetic}, Oxford Logic
  Guides, vol.~29, Oxford University Press, 1994, Ch. IX of [Sh:g].

\bibitem[She09]{Sh:300a}
\bysame, \emph{{Universal Classes: Stability theory for a model}}, 2009, Ch. V
  of [Sh:i].

\bibitem[Szm49]{Szm48}
W.~Szmielew, \emph{Decision problem in group theory}, Proceedings of the tenth
  international Congress of Philosophy (Amsterdam), Library of the Tenth
  International Congress of Philosophy, vol.~1, North-Holland, 1949,
  pp.~763--766.

\bibitem[Szm55]{Szm55}
\bysame, \emph{Elementary properties of abelian groups}, Fund. Math.
  \textbf{41} (1955), 203--271.

\end{thebibliography}
